\newif\ifpdf
\ifx\pdfoutput\undefined
\pdffalse 
\else
\pdfoutput=1 
\pdftrue \fi

\newif\iffinal
\finaltrue 

\documentclass[reqno,twoside,11pt]{amsart}
\iffinal\else\usepackage[notref,notcite]{showkeys}\fi
\usepackage{amsmath}
\usepackage{amsfonts}
\usepackage{amssymb}
\usepackage{verbatim}
\usepackage{epsfig}
\usepackage{setspace}

\IfFileExists{epsf.def}{\input epsf.def}{\usepackage{epsf}}

\IfFileExists{myowntimes.sty}{\usepackage{myowntimes}\usepackage{mathrsfs}}
	{\usepackage{mathpazo}\usepackage{mathrsfs}}

\DeclareFontFamily{OT1}{eusb}{} \DeclareFontShape{OT1}{eusb}{m}{n} {<5> <6> <7> <8> <9> <10> <11> <12> <14.4> eusb10}{}
\DeclareMathAlphabet{\eusb}{OT1}{eusb}{m}{n}

\DeclareFontFamily{OT1}{eusm}{} \DeclareFontShape{OT1}{eusm}{m}{n} {<5> <6> <7> <8> <9> <10> <11> <12> <14.4> eusm10}{}
\DeclareMathAlphabet{\eusm}{OT1}{eusm}{m}{n}

\DeclareFontFamily{OT1}{eufm}{} \DeclareFontShape{OT1}{eufm}{m}{n} {<5> <6> <7> <8> <9> <10> <11> <12> <14.4> eufm10}{}
\DeclareMathAlphabet{\mathfrak}{OT1}{eufm}{m}{n}

\DeclareFontFamily{OT1}{fraktura}{}
\DeclareFontShape{OT1}{fraktura}{m}{n} {<5> <6> <7> <8> <9> <10> <11> <12> <13> <14.4> [1.1] eufm10}{}
\DeclareMathAlphabet{\fraktura}{OT1}{fraktura}{m}{n}

\DeclareFontFamily{OT1}{cmfi}{} \DeclareFontShape{OT1}{cmfi}{m}{n} {<5> <6> <7> <8> <9> <10> <11> <12> <13> <14.4> [0.9] cmfi10}{}
\DeclareMathAlphabet{\cmfi}{OT1}{cmfi}{b}{n}

\DeclareFontFamily{OT1}{cmss}{} \DeclareFontShape{OT1}{cmss}{m}{n} {<5> <6> <7> <8> <9> <10> <11> <12> <13> <14.4> cmss10}{}
\DeclareMathAlphabet{\cmss}{OT1}{cmss}{m}{n}

\newtheoremstyle{thm}{1.5ex}{1.5ex}{\itshape\rmfamily}{} {\bfseries\rmfamily}{}{2ex}{}

\newtheoremstyle{def}{1.5ex}{1.5ex}{\slshape\rmfamily}{} {\bfseries\rmfamily}{}{2ex}{}

\newtheoremstyle{rem}{1.3ex}{1.3ex}{\rmfamily}{} {\itshape}
{} {1.5ex}{}

\theoremstyle{thm}
\newtheorem{theorem}{Theorem}[section]
\newtheorem{lemma}[theorem]{Lemma}
\newtheorem{proposition}[theorem]{Proposition}

\newtheorem*{Main Theorem}{Main Theorem.}
\newtheorem{corollary}[theorem]{Corollary}
\newtheorem*{special theorem}{Lindeberg-Feller Theorem for Martingales}

\theoremstyle{def}

\theoremstyle{rem}
\newtheorem{remark}{{\itshape Remark}}[]

\numberwithin{equation}{section}

\renewcommand{\section}{\secdef\sct\sect}
\newcommand{\sct}[2][default]{%
\refstepcounter{section}
\addcontentsline{toc}{section}{{\tocsection {}{\thesection}{\!\!\!\!#1\dotfill}}{}}
\vspace{0.7cm}
\centerline{\scshape\thesection.\ #1} \nopagebreak \vspace{0.2cm}}
\newcommand{\sect}[1]{%
\vspace{0.4cm} \centerline{\large\scshape\rmfamily #1}
\vspace{0.2cm}}

\renewcommand{\subsection}{\secdef\subsct\sbsect}
\newcommand{\subsct}[2][default]{\refstepcounter{subsection}
\addcontentsline{toc}{subsection}
{{\tocsection{\!\!}{\hspace{1.2em}\thesubsection}{\!\!\!\!#1\dotfill}}{}}
\nopagebreak\vspace{0.45\baselineskip} {\flushleft\bf
\thesubsection~\bf #1.~}
\\*[3mm]\noindent
\nopagebreak}
\newcommand{\sbsect}[1]{\vspace{0.1cm}\noindent
\textbf{#1.~}\vspace{0.1cm}}

\renewcommand{\subsubsection}{%
\secdef \subsubsect\sbsbsect}
\newcommand{\subsubsect}[2][default]{%
\refstepcounter{subsubsection}
\addcontentsline{toc}{subsubsection}{{\tocsection{\!\!}
{\hspace{3.05em}\thesubsubsection}{\!\!\!\!#1\dotfill}}{}}
\nopagebreak
\vspace{0.15\baselineskip} \nopagebreak {\flushleft\rmfamily
\itshape\thesubsubsection
\ \rmfamily #1\/.}\ }
\newcommand{\sbsbsect}[1]{\vspace{0.1cm}\noindent
\rmfamily \itshape
\arabic{section}.\arabic{subsection}.\arabic{subsubsection} \
\sffamily #1\/.\ }

\iffinal

\else

\fi

\renewcommand{\caption}[1]{%
\vglue0.5cm
\refstepcounter{figure}
\begin{minipage}{0.9\textwidth}\small {\sc Figure~\thefigure. }#1\end{minipage}}

\newcommand{\CC}{\mathcal C}
\newcommand{\DD}{\mathcal D}
\newcommand{\EE}{\mathcal E}
\newcommand{\FF}{\mathcal F}
\newcommand{\GG}{\mathcal G}
\newcommand{\HH}{\mathcal H}

\newcommand{\LL}{\mathcal L}
\newcommand{\MM}{\mathcal M}
\newcommand{\NN}{\mathcal N}

\newcommand{\RR}{\mathcal R}

\newcommand{\VV}{\mathcal V}
\newcommand{\WW}{\mathcal W}
\newcommand{\XX}{\mathcal X}

\newcommand{\E}{\mathbb E}

\newcommand{\N}{\mathbb N}

\def\myffrac#1#2 in #3{\raise 2.6pt\hbox{$#3 #1$}\mkern-1.5mu\raise 0.8pt\hbox{$#3/$}\mkern-1.1mu\lower 1.5pt\hbox{$#3 #2$}}

\newcommand{\Cov}{\text{\rm Cov}}

\newtheorem{thm}{Theorem}[section]
\newtheorem{lmm}[thm]{Lemma}

\newcommand{\ee}{\mathbb{E}}

\newcommand{\mf}{\mathcal{F}}

\newcommand{\ra}{\rightarrow}
\newcommand{\rr}{\mathbb{R}}

\title[CLTs for the SK model]{Central Limit Theorems for the Energy Density in the Sherrington-Kirkpatrick Model}
\author[S.~Chatterjee and N.~Crawford]
{Sourav Chatterjee and Nicholas Crawford}
\thanks{{\tt
    email:sourav@stat.berkeley.edu}, Supported by a Sloan Research Fellowship and NSF grant DMS 0707054}  
\thanks{{\tt
    email:crawford@stat.berkeley.edu}, Supported in part by DOD ONR grant
  N0014-07-1-05-06}
\begin{document}
\thanks{\hglue-4.5mm\fontsize{9.6}{9.6}\selectfont\copyright\,2008 by S.~Chatterjee and N.~Crawford. Reproduction, by any means, of the entire
article for non-commercial purposes is permitted without charge.\vspace{2mm}}
\maketitle

\vspace{-5mm}
\centerline{\textit{Department of Statistics, University of California at Berkeley}}

\vspace{-2mm}

\begin{abstract}
In this paper we consider central limit theorems for various macroscopic observables in the high temperature region of the Sherrington-Kirkpatrick spin glass model.  With a particular focus on obtaining a quenched central limit theorem for the energy density of the system with non-zero external field, we show how to combine the mean field cavity method with Stein's method in the quenched regime.  The result for the energy density extends the corresponding result of Comets and Neveu in the case of zero external field.
\end{abstract}

\section{Introduction}
The study of mean field disordered systems has lately seen much interest from the theoretical probability and (mathematical) statistical physics communities.  In this paper we reinvestigate the general problem of proving central limit theorems for macroscopic observables in such systems.  For a reasonable family of models, the book \cite{Talagrand-book} shows how to obtain such theorems in a direct way, by computing all limiting moments of the random variables in question (in particular we refer to Sections 2.5, 2.6, 2.7, 3.5, 3.6 and 5.10).  Vaguely, the method rests on having \textit{a priori} control of some fundamental order parameter (in the high temperature regime).

For the specific model we consider here, the Sherrington-Kirkpatrick model \cite{SK}, among the notable rigorous contributions we mention \cite{ASS, AA, CM, Guerra-Ton-1, Guerra-Ton-2, Guerra, Guerra-2, Talagrand-paper-1}, culminating in the verification of the Parisi formula \cite{Talagrand-paper}.  The SK model is defined via the Gibbs measure on spin configurations $\sigma \in \Sigma_N := \{-1, 1\}^N$ with mean field interaction given by the Hamiltonian
\begin{equation*}
H_N(\sigma) = \sum_{1 \leq i<j \leq N} \frac{1}{\sqrt N} g_{i,j} \sigma_i \sigma_j + h \sum_{i=1}^N \sigma_i.
\end{equation*}
The couplings $g_{i,j}$ are assumed to be independent Gaussian variables with mean $0$ and variance $1$; $h \in \mathbb R$ denotes the strength of the external field.  In other words, each spin configuration $\sigma \in \Sigma_N$ is chosen with probability
\begin{equation*}
P(\sigma) \propto
e^{\beta H_N(\sigma)}.
\end{equation*}
The parameter $\beta$ denotes the inverse temperature and note that we have omitted the minus sign from the exponent for convenience.

In the case of the Sherrington-Kirkpatrick model, the role of the order parameter mentioned in the first paragraph is played by the overlap $R_{1,2}$:
Let $\sigma^1, \sigma^2 \in \{-1,1\}^N$ denote a pair of spin configurations.  The overlap between $\sigma^1, \sigma^2$ is given by
\begin{equation*}
R(\sigma^1, \sigma^2) = R_{1,2}: = \frac{1}{N}\sum_{i=1}^N \sigma_i^1 \sigma_i^2.
\end{equation*}

Control of $R_{1,2}$ is obtained by relating the full system of $N$ spins to a system in which one of the particles has been decoupled from the other $N-1$ in a `smart' way via the cavity method.
Among other consequences, the cavity method allows the explicit computation of all moments for $R_{1,2}$ when properly centered and scaled, and thus proves quenched and "quenched average" CLTs for $R_{1,2}$ via the method of moments  (for a definition of quenched and quenched average, see Section \ref{S:notes}).

Our goal in the present work is two fold:  First, from a specific viewpoint, we explore the limiting behavior of
\begin{equation*}
H_N=H_N(\sigma) = \sum_{1 \leq i<j \leq N} \frac{1}{\sqrt N} g_{i,j} \sigma_i \sigma_j + h \sum_{i=1}^N \sigma_i
\end{equation*}
under the quenched and "quenched averaged" Gibbs distributions.  
In the $h=0$ case, this problem was studied previously by Comets and Neveu in the beautiful work \cite{CM} and later was re-derived using Stein's method in \cite{Chatterjee}.  
Let us recall for the reader the relevant result:
\begin{proposition}[Proposition 5.2 of \cite{CM}]
\label{P:CN}
Let $h=0$ and $\beta< 1$.  Then as $N \rightarrow \infty$, the law of 
\begin{equation}
\HH_N:= (N- 1)^{-\frac{1}{2}}[ H_N(\sigma) - (N - 1)\beta/2] 
\end{equation}
under the quenched Gibbs probability distribution weakly converges in probability to   
the law of a centered Gaussian with variance $\frac{\beta^2}{2}$. More precisely, we have
\begin{equation}
\lim_{N \rightarrow \infty} \langle \exp[\mu \HH_N] \rangle= \exp(\beta^2 \mu^2/4)
\end{equation}
for every real $\mu$, where the convergence occurs in probability with respect to the i.i.d. Gaussian couplings. 
\end{proposition}

In a broader context, the present work may be seen as a test case for a general approach which strengthens, unifies and extends the CLTs from \cite{Talagrand-book}.  In this respect, we address two issues:  first, the 'method of moments' only provides distributional convergence in the weak sense (and no rate of convergence in any associated metric).  Second, an inherent feature of this approach is that a CLT for any \textit{other} macroscopic observable requires the computation of joint  moments with overlaps, creating extra work each time we consider some new observable. In particular, we point out that the condition for having central limit theorems for a variety of observables rests on the invertibility of a single matrix.  It seems reasonable to believe our method extends to other systems, however below we limit considerations to the SK model.

It is worthwhile to compare Proposition \ref{P:CN} to our main result, Theorem \ref{T:Sum} below.  First, our convergence result for the \textit{quenched average} law of $H_N$ is in the stronger Wasserstein distance as opposed to weak convergence (which is implicit in the above result).  Also, as one can glean by a simple comparison of the general formulas for the quenched average and quenched variances $\sigma_A^2$ and $\sigma_Q^2$ respectively, the $h \neq 0$ case involves considerably more intricate computations.  Note further that in general $\sigma_A \neq \sigma_Q$, which is however the case if $h =0$.  Finally Proposition \ref{P:CN} implies that if $h=0$, then up to terms of the form $o(N^{\frac 12})$, $\HH_N$ is  centered under the \textit{quenched} Gibbs state.  The analogous statement for $h \neq 0$ is not true.

As a general rule, when $h \neq 0$ producing results for the SK model is harder than when $h=0$ (attention was brought to this fact by Talagrand in \cite{Talagrand-paper-1}).
From a structural point of view, the difficulties associated with the case $h \neq 0$ may be summarized by the fact that under the quenched Gibbs state, the pair of overlaps $R_{1,2}$, $R_{2,3}$ are uncorrelated when $h=0$ whereas this is not the case for $h \neq 0$ (even after subtracting off their common mean).

The remainder of this paper is organized as follows.  In the next subsection we introduce much of the notation and mention the interpolation we shall use explicitly.  In addition, we give a brief summary of Stein's Method, which is the main tool we use besides the cavity method.  The reader wishing a more detailed introduction should consult \cite{BC}.  Section 1.3 gives a rundown of the explicit results appearing in the paper.  In Section 2 we review some technical facts from \cite{Talagrand-book} and extend them to give our main estimate, Corollary \ref{C:overlap}.  In Section 3 we give a self contained proof of a quenched average CLT for the internal energy (i.e. the quadratic portion of $H_N$).  Sections 4 and 5 provide the proof of our main result.  Finally, Section \ref{S:Proofs} contains a number of the calculations used in the rest of the paper.

\subsection{Notations}
\label{S:notes}
The consideration of quenched Gibbs states below provides a number of complications.  First, measuring the size of a quenched average is most conveniently addressed through the concept of \textit{replicas}:   Let us fix a realization of $\{g_{i, j}\}$.  For each $\sigma \in \Sigma_N^n$
\begin{equation*}
H_N^n(\sigma)= \sum_{r=1}^n H_N(\sigma^r), \quad \text{ with } \sigma^r \in \Sigma_N \: \forall \: 1 \leq r\leq n.
\end{equation*}
The quenched Gibbs state corresponding $H_N^n$ is denoted by $\langle \cdot \rangle$.  We denote the $n$-replica quenched average Gibbs state by $\nu\left(\cdot\right)=\E\left[\langle \cdot \rangle \right]$.  There should be no confusion here since both $\langle\cdot\rangle$ and $\nu$ define consistent families as $n$ varies.

Let $q_2$ denote the solution to the equation
\begin{equation*}
\label{Eq:MF}
q_2= \mathbb E\left[ \tanh^2 \left(\beta \sqrt{q_2} z +h \right)\right]
\end{equation*}
where $z$ is a standard Gaussian random variable.
A result due to Guerra \cite{Guerra-2} and Lata{\l}a \cite{Latala} shows that there is a unique solution for $q_2$ whenever $h>0$.  We shall use the notation
\begin{equation}
\label{Eq:qp}
q_p = \mathbb E\left[ \tanh^p \left(\beta \sqrt{q_2} z +h \right)\right].
\end{equation}

A fundamental result, due to Frohlich and Zegarlinski \cite{FZ} without rates of convergence  and
Talagrand \cite{Talagrand-paper-1} with rates, is that for each $h \in \mathbb R$ and at high enough temperature,
\begin{equation*}
\label{Eq:High-Temp-1}
\E \left[\left\langle \left(R_{1,2} - q_2\right)^2 \right\rangle\right]  \leq \frac{C}{N}
\end{equation*}
for some constant $C>0$ as $N$, the number of spins in the system, tends to $\infty$.  In fact, we shall assume that $(\beta, h)$ satisfies
\begin{equation}
\label{Eq:High-Temp}
\E \left[\left\langle \left(R_{1,2} - q_2\right)^6 \right\rangle\right]  \leq \frac{C}{N^3}
\end{equation}
We remark that according to \cite{Talagrand-book} Section 2.5, there exists a $\beta_0 > 0$ independent of $N, h$ so that \eqref{Eq:High-Temp} holds for all $0 \leq \beta \leq \beta_0$ (actually much more was proved: If $\beta \leq \beta_0$, the random variable $N \left(R_{1,2}-q_2\right)^2$ has finite exponential moments in a neighborhood of $0$).

Our study is facilitated by viewing $H_N$ as composed of two terms
\begin{equation*}
H_N= E_N + N h M_N
\end{equation*}
where
\begin{align*}
E_{N}(\sigma)= & \sum_{1 \leq i<j \leq N} \frac{1}{\sqrt N} g_{i,j} \sigma_i \sigma_j\\
M_N = & \frac{1}{N} \sum_{j=1}^N \sigma_j
\end{align*}
denote the internal energy and magnetization of the system respectively.  Since we are interested in the fluctuations of these variables, it is convenient to define the normalized quantities
\begin{align*}
\EE_N &=\frac{1}{\sqrt{N}} E_N - \frac{\beta \sqrt{N}}{2} \left(1-q_2^2\right),\\
\MM_N &= \sqrt{N}\left[M_N- q_1\right],\\
\RR_{1,2} &=\sqrt{N}\left[R_{1,2}- q_2\right],\\
\HH_N&=\EE_N+ \MM_N.
\end{align*}

Next let us review the cavity interpolation.  Given the $n$-replica quenched average Gibbs measure $\nu$, we define the measure $\nu_t$ as follows.  Let $\{z_r\}_{r=1}^n$ denote a sequence of standard Gaussian variables independent of the coupling constants $\{g_{i,j}\}_{1 \leq i < j \leq N}$.  Isolating the last spin of each replica with the notation $\varepsilon^r$, let $\langle \cdot \rangle_t$ denote the quenched Gibbs state with Hamiltonian defined by
\begin{equation*}
H_{N, t}^n = \sqrt{t} \sum_{1 \leq r \leq n} \: \sum_{1 \leq i < j \leq N-1}  \: \frac{1}{\sqrt{N}}g_{i,j} \sigma^{r}_i \sigma^{r}_j + \sqrt{1-t}\sqrt{q_2} z_r \varepsilon^{r}.
\end{equation*}
Then we let
\begin{equation}
\label{Eq:nu-t}
\nu_t(f):= \mathbb E \left[\langle f \rangle_t\right].
\end{equation}
In particular $\nu_0$ decouples the replica spins corresponding to the last site from the remainder of the spin system.  We use the notation $R_{k,k'}^-$ to denote $R_{k, k'}- \varepsilon^k \varepsilon^{k'}/N$ and extend this notation in an analogus manner to macroscopic variables such as $M_N$, $\MM_N$, $E_N$, $H_N$ etc.  

An object of importance below are the $\textit{local fields}$: for an index $r \in [n]$, define
\begin{equation*}
\ell^r_N = \frac{1}{\sqrt N} \sum_{i< N} g_{i, N} \sigma^r_i. 
\end{equation*}

As will become clear below, for the purposes of a quenched CLT, we will need to give an extension on the usual bounds for the Stein characterizing equation of a random vector.
Let  $\| \vec x \|_2$ denote the Euclidean length of $\vec x \in \mathbb R^{2n}$ and let $G(\vec x)$ be a polynomial in the components of $\vec x$ with total degree $d$.
Obviously, there exists a constant $C_G>0$ so that
\begin{equation*}
\left|\frac{\partial}{\partial x_{i_1}} \cdots \frac{\partial}{\partial x_{i_m}} G(x)\right| \leq C_G \left(1+\|x\|_2\right)^d
\end{equation*}
for each $m$-tuple $(i_1, \dotsc, i_m), \in \N^m$, the bound holding uniformly in $m$.  Let $C^1_b(\mathbb R^d)$ denote the set of functions $f: \mathbb R^d \rightarrow \mathbb R$ which are bounded and have one continuous derivative.
Let $\FF=\{ f \in C^1_b(\mathbb R): f' \text{ is globally Lipschitz}\}$ and define a norm on $\FF$ by $\|f\|_{\FF} = \|f\|_{\infty}+ \|f'\|_{\infty} + Lip(f')$.  More generally, let
\begin{equation}
\label{Eq:F}
\FF_{2n}=\left\{ F \in C^1_b(\mathbb R^{2n})\;:\; \nabla F \text{ is globally Lipschitz}\right\}.
\end{equation}
Our case study will revolve around the vector 
\begin{equation}
\label{Eq:XX}
\XX^n_N=(\EE^1_N,\dotsc \EE^n_N,  \MM^1_N, \dotsc, \MM_N^n )
\end{equation}
where the superscript refers to the replica under consideration.

Let $F\in \FF_{2n}$ and let $G: \mathbb R^{2n} \rightarrow \mathbb R$ a multivariable polynomial.
To keep track of the error dependence in our calculations, let us introduce the notation $Er(F, G)$ to mean any term involving $F$ and $G$ which can be bounded above by
\begin{multline*}
\left|Er(F, G)\right| \leq \left\{1+ \|F\|_{\infty} + \|F'\|_\infty + Lip(\nabla F) \right\}\times\\
\left\{1+C_G \nu\left(\left(1+\|\XX^{n}\|_2\right)^{2d}\right)^{\frac{1}{2}}  \right\}\frac{C}{\sqrt{N}}.
\end{multline*}
Here and below, the constant $C> 0$ will stand for a generic constant which may depend on  the number of replicas $n$, the inverse temperature $\beta$ and the external field $h$, but will not depend on $F, G$ or $N$.  The value of this constant may (and will) change from line to line.

For the convenience of the reader, let us sketch Stein's method (the reader may find a systematic introduction to the method in \cite{BC}).  The method exploits two basic ideas.  The first is that one can define distances between the distributions of random variables by optimizing over classes of test functions.  In this paper for example, we use the Wasserstein and L\'{e}vy metrics.

The other idea is that distributions satisfy `functional identities' which can be used to characterize them.  Let $z_{\eta}$ denote a Gaussian with mean $0$ and variance $\eta^2$. The well known integration-by-parts identity for Gaussian variables says
\begin{equation}
\label{Eq:GIP}
\mathbb E\left[ z_{\eta} F(z_{\eta})\right]= \eta^2 \mathbb E\left[F' (z_{\eta})\right].
\end{equation}
Now, given a test function $u$, suppose we solve the ordinary differential equation
\begin{equation*}
xF(x)- \eta^2 F'(x) = u(x)- \mathbb E\left[u(z_{\eta})\right],
\end{equation*}
then we obtain
\begin{equation*}
\left|\mathbb E \left[u(X)\right]- \mathbb E\left[u(z_{\eta})\right]\right|= \left|\mathbb E \left[XF(X)- \eta^2 F'(X)\right]\right|.
\end{equation*}
Hence, if we can show that $X$ approximately satisfies \eqref{Eq:GIP}, this leads to bounds on the distance of the distribution of $X$ to a Gaussian distribution in the appropriate sense.

Before stating our main results, let us recall the definitions of the aforementioned metrics:
Suppose $X,Y$ are random variables on $\mathbb R$  with associated distributions $\mu_X, \mu_Y$ and cumulative distribution functions $F_X, F_Y$, we define the Wasserstein distance by
\begin{equation*}
\WW_1(X, Y)= \sup_{\left \{\underset{Lip(u) \leq 1 }{u: \mathbb R \rightarrow \mathbb R \text{ s.t.}}\right\}} \left|\mathbb E[u(X)] - \mathbb  E[u(Y)]\right|.
\end{equation*}
Further, recall that weak-$*$ convergence on the space of measures on $\mathbb R$, when restricted to probability measures, is metrizable by the L\'{e}vy metric $\rho$: 
\begin{equation}
\label{Levy}
\rho(X, Y)= \inf\{\epsilon \geq 0: F_X(x- \epsilon) - \epsilon \leq F_Y(x) \leq F_X(x+\epsilon) +\epsilon \text{ for all $x \in \mathbb R$}\}.
\end{equation}

\subsection{Results}
\label{S:Results}
Let us begin by stating a CLT for $E_N$ under the quenched average measure $\nu$.  The proof of this result provides a template we shall use in the more complicated setting below.
Recall the definition of $q_p$ and $\nu$ from Section \ref{S:notes}.  We use the notation
\begin{align*}
a=& (1-q_2)^2\\
b=& 2q_2+ q_2^2 - 3q_4\\
c=& 1-6q_2-q_2^2+6q_4.
\end{align*}
Let us define the standard deviation $\sigma_A$ by
\begin{equation}
\label{sig_a}
\sigma_a^2=\frac{1}{2} + \frac{\beta^2 q_2 \left(a - \beta^2 (b^2 + ac) \right)}{(1- \beta^2a)(1-\beta^2c)+ \beta^4 b^2}.
\end{equation}

\begin{theorem}
\label{T:annealed}
Suppose that $(\beta, h)$ satisfies the high temperature condition \eqref{Eq:High-Temp}.  Then for any function $f\in \FF$, under the quenched average measure $\nu$
\begin{equation*}
\nu\left(\EE_Nf\left(\EE_N\right)\right)=\sigma_a^2\nu\left(f'\left(\EE_N\right)\right)+ \|f\|_{\FF} CN^{-1/2}.
\end{equation*}
\end{theorem}
An easy consequence of this result in combination with Stein's method is a bound on the Wasserstein distance from the law of $\EE_N$ under $\nu$ to a Gaussian random variable.

Let us recall the basic result from the body of work known as Stein's Method which allows us to convert Theorem \ref{T:annealed} into a CLT with quantitative bounds.
\begin{lemma}[See \cite{BC}]
\label{L:Stein}
Let $g: \mathbb R \rightarrow \mathbb R$ be a Lipschitz continuous function with Lipschitz constant $L$.  Suppose that $f: \mathbb R \rightarrow \mathbb R$ solves the ordinary differential equation
\begin{equation}
\label{eqOU}
f'(x)- xf(x) = g(x)- \mathbb E\left[g(z_{\sigma=1})\right].
\end{equation}

Then
\begin{equation}
\|f\|_{\infty} \leq L, \quad, \|f'\|_{\infty} \leq\sqrt{\frac {2}{ \pi}} L, \quad \|f''\|_{\infty} \leq2 L.
\end{equation}
The last inequality is understood to mean that $f'$ is Lipschitz continuous with constant bounded by $2L$.
\end{lemma}

\begin{corollary}
\label{C:WassAnn-1}
Suppose that $(\beta, h)$ satisfies the high temperature condition \eqref{Eq:High-Temp}.   Consider the random variable $\EE_N$ under the quenched average measure $\nu$.  Let $z_{\sigma_A}$ denote a normal random variable with mean $0$ and variance $\sigma_A^2$.
Then we have
\begin{equation*}
\WW_1(\EE_N, z_{\sigma_A}) \leq CN^{-1/2}.
\end{equation*}
\end{corollary}

\begin{remark}
Below (Theorem \ref{T:Sum}) we go further and derive a quenched average CLT for $\HH_N$, and use this information to obtain the behavior of the distribution for $\langle \HH_N \rangle$, for which, interestingly, a Stein characterizing equation does not seem easily accessible.
\end{remark}

Our strategy for the derivation of Theorem \ref{T:annealed} runs roughly as follows:  Through a combination of the cavity method and Gaussian integration by parts we reduce the identification of a functional identity between $\nu\left(\EE_Nf\left(\EE_N\right)\right)$ with some scalar multiple of $\nu\left(f'\left(\EE_N\right)\right)$ to a characterization of the interaction between the overlap $R_{1,2}$ and the variable $\EE_N$ through the quantity
\begin{equation*}
\beta q_2 \nu\left(\RR_{1,2}f \left(\EE_N\right)\right).
\end{equation*}
Using the cavity method, we derive a pair of (approximate) linear equations involving this quantity and the pair
\begin{equation*}
\beta q_2 \nu\left(\RR_{2,3}f \left(\EE_N\right)\right) \: \: \text{and } \nu\left(f'(\EE_N)\right).
\end{equation*}
Solving these equations in terms of $\nu\left(f'(\EE_N)\right)$ identifies the Stein characterizing equation.

Let us next formulate a summary of our main findings.  Recall that weak convergence for the space of probability measures on $\mathbb R$ is metrizable via the L\'{e}vy metric $\rho$ \eqref{Levy}.
\begin{theorem}[The Full Picture]
\label{T:Sum}
Suppose that $(\beta, h)$ satisfies the conditions of Theorem \ref{T:Main}.
Then:
\begin{enumerate}
\item
\label{I:1}
Consider the random variable $\HH_N$ under $\nu$.  There exists a variance $\sigma_A^2$ depending only on $(\beta, h)$ so that
\begin{equation*}
\WW_1(\HH_{N}, z_{\sigma_A}) \leq C N^{-1/2}.
\end{equation*}
\\
\item
\label{I:2}
Let $\LL_N$ denote the random variable $\HH_N-\langle \HH _N\rangle$ under the quenched Gibbs distribution. With $\sigma_Q^2$ as below in Corollary \ref{C:Quenched}, for all $\epsilon > 0$
\begin{equation*}
\lim_{N \rightarrow \infty} \mathbb P\left(\rho\left(\LL_N, z_{\sigma_Q} \right) \geq \epsilon \right) = 0.
\end{equation*}
 \\
\item
\label{I:3}
Finally consider the law of $\langle \HH _N\rangle$ under the Gaussian probability measure.  We have
\begin{equation*}
\lim_{N \rightarrow \infty} \rho(\mu_{\langle \HH _N\rangle}, z_{\sqrt{\sigma_A^2 - \sigma_Q^2}}) =0.
\end{equation*}
\end{enumerate}
\end{theorem}
We next provide a brief sketch of what is to come, in particular regarding \eqref{I:2} of the previous theorem.\\

\textit{Sketch of Quenched Central Limit Theorem:}\newline
Recall the definition of $\XX_N^n$ \eqref{Eq:XX}.
To prove the quenched CLT, we proceed as follows. The main computational ingredient, stated below as Theorem \ref{T:Main} is to derive multivariable quenched average functional identities for the vector $\XX_N^n$.  This derivation is an elaboration of Theorem \ref{T:annealed}.  Stein's Method bounds then give Corollary \ref{Eq:AMCLT}, which in particular implies an \textit{quenched average} CLT for $\HH_N$.

Next, we use replicas to turn this set of quenched average functional identities into a \textit{quenched} functional identity by replicating the spin system.  By this we mean elaborate use of the following basic observation:  Given a function $f: \mathbb R \rightarrow \mathbb R$, suppose we are interested in size of the quenched variance, $\langle(f(\HH_N)- \langle f(\HH_N)\rangle)^2\rangle$.  This is itself a random variable, but we may estimate it by taking expectations:
\begin{equation*}
\mathbb E\left[\langle(f(\HH_N)- \langle f(\HH_N)\rangle)^2\rangle^2\right].
\end{equation*}
Then we may always represent this expression in terms of a family of $6$ replicas:
\begin{multline*}
\mathbb E\left[\langle(f(\HH_N)- \langle f(\HH_N)\rangle)^2\rangle^2\right] \\
= \nu\left((f(\HH^1_N)- f(\HH^2_N))(f(\HH^1_N)- f(\HH^3_N))(f(\HH^4_N)-f(\HH^5_N))(f(\HH^4_N)- f(\HH^6_N))\right).
\end{multline*}
This allows to employ the previously derived quenched average functional identities for $\XX_N^n$.

The application of this idea to get $L^2$ bounds on the quenched Stein equation is stated as Corollary \ref{C:Quenched}.  
Part of the subtlety here is that under the quenched Gibbs state, $\HH_N$ is \textit{not} centered.  Moreover, it seems unclear how to derive functional identities directly for the centered variable $\HH_N - \langle \HH_N \rangle$ since $\langle \HH_N \rangle$ depends on the $\{g_{i, j}\}$ and $\sigma$ in a complicated way.  Thus our quenched identity is stated in terms of a functional equation for a Gaussian variable shifted by the quenched mean $\langle \HH_N \rangle$.

This approach also leads to slight complications since the solution $f$ to the Stein equation
\begin{equation}
xf(x) - \sigma^2 f'(x) - \mu f(x) = g(x) - \mathbb E\left[g(\sigma z + \mu) \right]
\end{equation}
depends nonlinearly on $\mu, \sigma$ (for us $\sigma$ is fixed, however).  This problem is dealt with through Lemma \ref{L:Sour}, which may be of some independent interest.

Finally, using the quenched average CLT for $\HH_N$ and the quenched CLT for $\HH_N$ (or equivalently $\HH_N - \langle \HH_N \rangle$) and a characteristic function argument, we derive a CLT for the variable $\langle \HH_N \rangle$, which provides a fairly comprehensive picture of the fluctuations of the energy density at high temperatures (see Theorem \ref{T:Sum}).  This completes our sketch.

Let us next formalize what is proved in the ensuing sections.
To be precise, we require a bit of notation.
Let
\begin{equation*}
\tilde {\mathfrak A}^{k , k'}_{r,r'}:= \nu_0\left((\varepsilon^k\varepsilon^{k'} -q_2) (\varepsilon^r \varepsilon^{r'}- q_2) \right).
\end{equation*}
Recall that under $\nu_0$ replicas associated to the final site decouple from the previous $N-1$ and so these entries can be expressed explicitly in terms of $\{1, q_2, q_4\}$, depending only on the number of indices in common.

Let $\mathfrak A$ denote the ${n+2 \choose 2} \times {n+2 \choose 2}$ matrix indexed by the ordered pairs $1 \leq r<r' \leq n+2,\; 1 \leq k<k'\leq n+2$ with entries
\begin{equation}
\label{Eq:Replica-Matrix}
\frak A^{k , k'}_{r,r'}:=
\begin{cases}
\tilde {\frak A}^{k,k'}_{r, r'} &\text{$r' \leq k' \leq n$ or $r' < k'=n+1$},\\
- n \tilde {\frak A}^{k,k'}_{r, r'}  &\text{$k' \leq n, r'= n+1$},\\
{n+1 \choose 2} \tilde {\frak A}^{k,k'}_{r, r'}  &\text{$k' \leq n, r= n+1$},\\
\tilde {\frak A}^{k,n+1}_{r, n+1}- (n+1)\tilde {\frak A}^{k,n+1}_{r, n+2}  &\text{$k' = n+1,\,  r'= n+1$},\\
-(n+1)\tilde {\frak A}^{k,n+1}_{n+1, n+2} + {n+2 \choose 2}(q_4- q^2_2)  &\text{$k'= n+1,\, r = n+1,\, r'= n+2$},\\
1-q_2^2 - 2(n+2)(q_2-q_2^2)+{n+3 \choose 2}(q_4-q_2^2) &\text{both pairs are $\{n+1, n+2\}$},\\
0 & \text{otherwise (i.e. $ r \leq n$, $r'=n+2$}).
\end{cases}
\end{equation}
Parenthetically, we remark that this matrix is closely related to the coefficients appearing in Lemma \ref{L:cavitation}.
Recall the definition of $\FF_{2n}$ \eqref{Eq:F}.  For any  $F\in \FF_{2n}$ and any multivariable 
polynomial $G: \mathbb R^{2n} \rightarrow \mathbb R$, let $GF(\XX_N^n)= G(\XX_N^n)F(\XX_N^n)$

To put into context the following result, recall that if $\vec \GG$ is a Gaussian vector on $\mathbb R^d$ with covariance matrix $\CC$, then for any sufficiently regular $F: \mathbb R^d \rightarrow \mathbb R$,
\begin{equation}
\mathbb E[\vec \GG F(\vec \GG)] = \CC \cdot \mathbb E[\nabla F(\vec \GG)]
\end{equation}

\begin{theorem}
\label{T:Main}
Suppose that $(\beta,h)$ satisfies the high temperature condition \eqref{Eq:High-Temp}.  Let $\beta$ be small enough so that $Id - \beta^2 \frak A$ is invertible.  Then there exists a positive semi-definite covariance matrix $\frak C: \mathbb R^{2n} \rightarrow \mathbb R^{2n}$ so that 
\begin{align}
\label{eq:express}
\nu\left(\EE^i GF\left(\XX^n_N\right)\right)&= \mathfrak C \cdot  \nu\left(\nabla GF\left(\XX_N^n\right)\right)_i+ Er(F, G),\\
\nu\left(\MM^i GF\left(\XX^n_N\right)\right)&=\frak C \cdot  \nu \left(\nabla GF\left(\XX_N^n\right)\right)_{n+i} + Er(F, G).
\end{align}
\end{theorem}

\begin{remark}
Implicit in the above theorem is the fact that $\frak C$ arises as the limiting covariance matrix of the vector $\XX^n_N$.  This follows by specializing $G$ to be one of the coordinate functions and $F$ to be the constant function $1$, however this is not how we identify $\frak C$.  Also, it is worth noting that by exchangeability of replicas under $\nu$, the entries of $\frak C$ take only $6$ distinct values.
\end{remark}
\begin{remark}
Our condition on the invertibility of $Id- \beta^2 \frak A$ is a bit unsatisfying, but certainly holds for $\beta$ small enough (independent of $h$).  We have not attempted to characterize precisely invertibility. Qualitatively, the result is important, since one can envision applying similar approaches for proving CLTs in other (mean field) spin glass models.
\end{remark}

Specializing Theorem \ref{T:Main} to linear functionals of $\XX^n_N$, we may apply the Stein's method machinery.
\begin{corollary}[quenched average Multivariate CLT]
\label{Eq:AMCLT}
Let $w \in \mathbb R^{2n}$ be fixed and suppose $f \in \FF$.  Let $\XX_w= w \cdot \XX^n_N$.  Then we have
\begin{equation*}
\nu\left(\XX_w f \left(\XX_w\right)\right)= w\cdot \frak C  w \:  \nu\left( f' \left( \XX_w \right)\right) + \|f\|_{\FF}C N^{-1/2}.
\end{equation*}
Consequently,
\begin{equation*}
\WW_1(\XX_w, z_{\sigma_w}) \leq CN^{- 1/2}
\end{equation*}
where
\begin{equation*}
\sigma_w = w \cdot \frak C w.
\end{equation*}
\end{corollary}
The proof of this corollary is analogous to that of Corollary \ref{C:WassAnn-1} and is omitted from the paper.

Next we reformulate Theorem \ref{T:Main} so as to obtain a quenched CLT for the energy density of the SK model.  By Theorem \ref{T:Main},
\begin{equation}
\sigma_Q^2 = \lim_{N \rightarrow \infty} \nu\left((\HH_N - \langle \HH_N \rangle)^2\right)
\end{equation}
exists.  Moreover, it may be given explicitly by 
\begin{multline}
\sigma_Q^2= 1- q_2^2 + 2\beta q_2(q_3-q_1) 
-2 \beta q_2 \left([I - \beta^2 \frak C]^{-1} \cdot \vec v^e\right)_{1, 3}\\
-2 (\beta^2 q_1(1-q_2) + \beta^2(q_3 - q_1 q_2) - 2\beta q_2) \left([I - \beta^2 \frak C]^{-1} \cdot \vec v^m\right)_{1, 3}
\end{multline}
where $v^e_{k, k'} = \vec{w}^{k, k'} \cdot (1, -1, 0, 0)$ and $\vec v^m_{k, k'} = \vec{w}^{k, k'} \cdot (0, 0, 1, -1) $ and the ${4 \choose 2}$ vectors $\vec w^{k, k'}$ are given explicitly:
For $1 \leq k < k' \leq 4$ and $1 \leq r \leq 4$, let $A(k, k', r), \: B(k, k',r)$ be defined by
\begin{equation*}
A(k,k',r):=\nu_0\left(\left(\varepsilon^{k} \varepsilon^{k'}- q_2\right) \varepsilon^{r}\right)
\end{equation*}
and
\begin{equation*}
B(k,k',r):=
\begin{cases}
\beta q_2 b & \text{ if $k , k' \neq r$}\\
\beta a & \text{ if $k=r$ or  $k' = r$}\\
\end{cases}
\end{equation*}
otherwise.
For each pair $\{k, k'\}$, we denote by $\vec{w}_{k, k'}\in \mathbb R^{2n}$ the vector
\begin{equation*}
\vec{w}_{k, k'} :=\left(B(k,k', 1), B(k,k', 2),A(k, k',1), A(k, k',2),\right).
\end{equation*}

\begin{corollary}
\label{C:Quenched}
Suppose that $(\beta, h)$ satisfies the conditions of Theorem \ref{T:Main}.  Then there exists a deterministic variance $\sigma_Q^2$ so that
\begin{equation*}
\mathbb E\left[\left( \langle\left(\HH_N - \langle\HH_N\rangle\right)^2 \rangle - \sigma_Q^2 \right)^2 \right] \leq CN^{-1/2}.
\end{equation*}
Moreover, for any function $f$ so that $f, f'\in \FF$, we have
\begin{equation*}
\mathbb E\left(\langle \HH_Nf\left(\HH_N\right)- \sigma^2_Q f' \left(\HH_N\right) - \langle\HH_N\rangle f(\HH_N) \rangle^2 \right) \\
\leq \left(\|f\|_{\FF}+ \|f'\|_{\FF}\right)^2 CN^{-1/2}.
\end{equation*}
\end{corollary}


\section{Preliminaries}
\label{S:Prelims}
We begin with two basic facts which will be used often below.
The key lemma, proved in \cite{Talagrand-book}, is as follows.
\begin{lemma}[Proposition 2.5.3 from \cite{Talagrand-book}]
\label{L:Cav-Bound}
Suppose $\beta_0$ satisfies the high temperature condition \eqref{Eq:High-Temp}.
Let $ \beta \leq \beta_0$.
Then for any $n \in \mathbb N$, there is a $K(n)>0$ such that for all $f :\Sigma^n_N \rightarrow \mathbb R$,
\begin{align}
&
\label{Cav-1}\left|\nu\left(f\right)- \nu_{0}\left(f\right)\right| \leq \frac{K\left(n\right)}{\sqrt{N}} \nu\left(f^2\right)^{\frac 12}\\
&\label{Cav-2}\left|\nu\left(f\right)- \nu_{0}\left(f\right) - \nu_0'\left(f\right)\right| \leq \frac{K\left(n\right)}{N} \nu\left(f^2\right)^{\frac 12}.
\end{align}
\end{lemma}
Here $\nu_0'$ denotes the derivative of $\nu_t$ with repsect to $t$ evaluated at $t=0$.
Introducing the notation $[m]=\{1, \cdots, m\} \subset \mathbb N$,  a direct calculation produces an expression for $\nu_{0}'(f)$ (see p. 77 from \cite{Talagrand-book}):
\begin{multline}
\label{Cav-Method}
\nu_{0}'(f)=
\beta^2 \sum_{1\leq \ell< \ell' \leq   n} \nu_0\left(\varepsilon^{\ell}\varepsilon^{\ell'}\left(R^-_{\ell, \ell'}-q_2\right)f\right)\\
- \beta^2 \sum_{1 \leq \ell \leq n} \nu_0\left(\varepsilon^\ell \varepsilon^{n+1}\left(R^-_{\ell, n+1}-q_2\right)f \right)\\
+  \beta^2 \frac{n\left(n+1\right)}{2}\nu_0\left(\varepsilon^{n+1}\varepsilon^{n+2} \left(R^-_{n+1, n+2}-q_2\right)f\right).
\end{multline}

This expression motivates the following general estimate obtained from the cavity method.  \begin{lemma}
\label{L:cavitation}
Let $S$ be a fixed finite subset of $\N$ and $p$ be a fixed integer.  Let $\eta : [p] \rightarrow \N$ be an injective function.
Suppose that $n$ is large enough so that $\eta\left([p]\right) \subseteq [n]$.  Then if $f^{-}: \Sigma_{N-1}^n \rightarrow \mathbb R$, we have
\begin{multline*}
\nu\left(\left(\prod_{l=1}^p \varepsilon^{\eta\left(l\right)}- q_p\right)f^{-}\right)=
\beta^2 \left(q_{p-2}- q_pq_2\right)\sum_{1\leq \ell < \ell' \leq p} \nu\left(f^-\left(R_{\eta(\ell), \eta(\ell')}-q_2\right)\right)\\
+\beta^2 \left(q_p- q_p q_2\right)\sum_{\underset{1 \leq \ell \leq p}{ \ell' \in [n] \backslash \eta\left([p]\right)}} \nu\left(f^-\left(R_{\eta(\ell), \ell'}-q_2\right)\right)\\
+ \beta^2  \left(q_{p+2}-q_pq_2\right)\sum_{\underset{\{\ell, \ell'\} \subset [n] \backslash \eta([p])}{\ell \neq \ell'}} \nu\left(f^-\left(R_{\ell, \ell'}-q_2\right)\right)\\
- \beta^2 n \left(q_p-q_pq_2\right)\sum_{1 \leq \ell \leq p} \nu\left(f^-\left(R_{\eta(\ell), n+1}-q_2\right)\right)\\
- \beta^2 n \left(q_{p+2}-q_pq_2\right)\sum_{ \ell \in [n] \backslash \eta([p])} \nu\left(f^-\left(R_{\ell, n+1}-q_2\right)\right)\\
+  \beta^2 \frac{n\left(n+1\right)}{2} \left(q_{p+2}-q_p q_2\right)\nu\left(f^-\left(R_{n+1, n+2}-q_2\right)\right) + \nu\left((f^-)^2\right)^{\frac{1}{2}}O\left(\frac{1}{N}\right)
\end{multline*}
where $O\left(\frac{1}{N}\right)$ denotes an expression bounded by $CN^{-1}$.
\end{lemma}
We relegate a proof of this statement to Section \ref{S:Proofs}.  It is a direct consequence of  \eqref{Eq:High-Temp}, \eqref{Cav-Method} and
Lemma \ref{L:Cav-Bound}.

The key points here are the expression of the righthand side in terms of $\nu$ and not $\nu_0$ and the fact that the error is of order $1/N$, which is important when considering the scaled overlaps $\RR_{\ell, \ell'}$.

Let $\VV^n =(\vec \sigma^1, \dotsc, \vec \sigma^n)$ and let
\begin{equation*}
\VV^n_i= \VV^n - ( \vec 0 , \dotsc, \vec \sigma^i, \dotsc, \vec 0).
\end{equation*}
Further, for each $F: \mathbb R^{Nn} \rightarrow \mathbb R$ let
$F^{(i)}(\VV) = F(\VV_i^n)$.
In the next result we denote $\nabla F(\VV)$ the gradient of $F$ evaluated at $\VV$ and let $\nabla ^2 F$ denote its Hessian.

We let $\vec \RR$ denote the ${n+2 \choose 2}$-tuple overlaps
\begin{equation*}
\vec \RR = \left(\RR_{1,2}, \cdots, \RR_{1,n+2}, \cdots,  \RR_{n+1, n+2}\right)
\end{equation*}
lexicographically ordered so that $(i, j) < (k, \ell)$ if $i< k$ or if $i=k$ and $j <\ell$.  If
$\mathbf 1 = (1, \dotsc, 1) \in \mathbb R^{{n+2  \choose 2}}$
let $\vec \sigma_i= (\vec \sigma^{r r'}_i)_{1\leq r< r' \leq n}: = (\sigma^r_i \sigma^{r'}_i)_{1\leq r< r' \leq n} - q_2 \bf 1$ denote the (lexicographically ordered) replica vector of spins evaluated on the $i$'th vertex.

Recall the definition of the matrix $\frak A$ from \eqref{Eq:Replica-Matrix}.
The following is a simple, but crucial corollary of the previous lemma.
\begin{corollary}
\label{C:overlap}
Suppose that the high temperature condition \eqref{Eq:High-Temp} holds.  Suppose that $F \in C^2(\mathbb R^{Nn})$.
Then entry-wise, we have the following estimate:
\begin{equation}
\label{Eq:overlap-a}
\nu\left(F(\VV^n)\left([Id - \beta^2 \frak A] \cdot \vec{\RR}\right) \right)= \frac{1}{\sqrt N} \sum_{i=1}^N\nu\left(\nabla F(\VV^n_i)\cdot \left(\VV^n - \VV^n_i\right)\vec \sigma_i \right)\\
+  \textrm{Rem $(F)$}
\end{equation}
where
\begin{multline*}
\textrm{Rem $(F)$}=
\frac{\beta^2}{N} \sum_{i=1}^N \nu\left(\left(F^{(i)} -   F\right) \frak A \cdot \vec \RR \right) \\+
 \frac{1}{\sqrt{N}} \sum_{i=1}^N \nu\left(\left(\VV^n - \VV^n_i\right) \cdot \left( \nabla^2 F (\VV_{0, i}) \cdot \left(\VV^n - \VV^n_i\right) \right)\vec \sigma_i \right)
\\
+ \left(\frac{1}{N} \sum_{i=1}^N \nu\left(\left(F^{(i)}\right)^2\right)^{\frac{1}{2}}\right) O\left(\frac{1}{\sqrt{N}}\right)
\end{multline*}
with $\VV_{0, i}=\lambda_0 \VV^n + (1-\lambda_0) \VV^n_i$ for some (random) $\lambda_0=\lambda_0(\VV^n, \VV^n_i)  \in [0,1]$.
\end{corollary}

\begin{remark}
There are two points worth making:
First, most functions $F$ we are interested in depend very weakly on any given site $i$; one should think in the simplest case of $F(\vec \sigma^1)= G\left(\MM^1_N\right)$.  The advantage of the above formulation, which we shall not pursue, is that one may consider functions like
$F\left(\vec \sigma\right): = f\left(\vec{u} \cdot \vec {\sigma}^1\right)$ for any $\vec u \in \mathbb S^{N-1}$ such that all coordinates of $\vec u$ are `small' in an appropriate sense.  Of course, $Rem (F)$ will be seen to be $O(\frac{1}{\sqrt N})$ in situations considered below.

Second, notice that the first term on the right hand side of \eqref{Eq:overlap-a} the argument $\VV_i^n$ does not depend on the $i$'th vertex, which means we may apply the first estimate from Lemma \ref{L:Cav-Bound} to obtain expressions in terms of $\nu(\nabla F(\VV^n_i))$.  The calculations then `close' by applying Taylor's theorem to express $\nu(\nabla F(\VV^n_i))$ as $\nu(\nabla F(\VV^n))$ plus some error.  This sets up the functional identities we are after.
\end{remark}

\section{The quenched average CLT for the Internal Energy}
\label{S:Annealed}
Throughout this section we work under the hypotheses of Theorem \ref{T:annealed}.
Assume for now that $f \in \FF$.
Let us begin with an elementary calculation using Gaussian integration by parts:
\begin{equation*}
\nu\left(\EE_N f \left(\EE_N\right)\right) = \frac{1}{2} \nu\left(f'\left(\EE_N\right)\right) -  \frac{\beta \sqrt{N}}{2}\nu\left(\left[R_{1,2}^2- q_2^2\right]f \left(\EE_N\right)\right) + O\left(\frac{\|f\|_{\FF}}{\sqrt{N}}\right).
\end{equation*}

It is now natural to write $R_{1,2}= R_{1,2}- q_2 + q_2$.  Expanding in the above expression we have
\begin{multline*}
- \frac{\beta \sqrt{N}}{2}\nu\left(\left[R_{1,2}^2-q_2^2\right]f \left(\EE_N\right)\right) = \\ \beta q_2 \nu\left(\RR_{1,2} f \left(\EE_N\right)\right) - \frac{\beta \sqrt{N}}{2} \nu\left(\left(R_{1,2}-q_2\right)^2 f \left(\EE_N\right)\right)
\end{multline*}
where $\RR_{1,2}$ denotes the centered, scaled overlap defined in Section \ref{S:notes}.
Because of the high temperature assumption \eqref{Eq:High-Temp}, we may bound the final summand by $\frac{C \beta\|f\|_\infty}{2 \sqrt{N}}$.
Hence the goal of our analysis will be to identify a formula (up to order \;$N^{- 1/2}$) for
\begin{equation*}
\beta q_2 \nu\left(\RR_{1,2} f \left(\EE_N\right)\right).
\end{equation*}

Though we do not appeal explicitly to Corollary \ref{C:overlap}, the next two lemmas essentially re-derive that result in explicit form.
Recall the constants $a, b, c$ defined above Theorem \ref{T:annealed}. We state our first lemma.
\begin{lemma}
\label{L:Point-One}
We have the identity
\begin{equation*}
\nu\left((1-\beta^2 a)\RR_{1,2}f (\EE_N)\right) =
-\beta^2 b\nu\left(\RR_{2,3}f (\EE_N)\right) +  \beta q_2 a\nu\left(f' (\EE_N)\right) \\+  O\left(\frac{\|f\|_{\FF}}{\sqrt{N}}\right).
\end{equation*}
\end{lemma}

\textit{Proof.}
Using symmetry, we have
\begin{equation}
\nu\left(\RR_{1,2}f (\EE_N)\right) =  \nu\left(\sqrt{N} \left(\varepsilon^1 \varepsilon^2- q_2\right)f\left(\EE_N\right)\right).
\end{equation}

We decouple the argument of $f$ from the last spin variable by writing
\begin{equation*}
\EE_N= \EE_N^- + \frac{\varepsilon^1 \ell^1_N}{\sqrt N}.
\end{equation*}
An easy calculation shows that $\nu\left((\ell^1_N)^2\right) \leq C$ for some constant $C>0$ independent of $N$ thus we may use the Taylor expansion to linearize $f$ around $\EE_N^-$.
We obtain the expression
\begin{multline}
\label{Eq:Start-1}
\sqrt{N}  \nu\left(\left(\varepsilon^1 \varepsilon^2- q_2\right)f\left(\EE_N\right)\right)= \nu\left(\sqrt{N} \left(\varepsilon^1 \varepsilon^2- q_2\right) f\left(\EE_N^-\right)\right)\\+
  \nu\left( \left(\varepsilon^1 \varepsilon^2- q_2\right) \varepsilon^1 \ell_N f'\left(\EE_N^-\right)\right) + O\left(\frac{\|f\|_{\FF}}{\sqrt{N}}\right).
\end{multline}

Note that $\EE_N^-$ does not involve replicas of spins associated to the $N^{th}$ site, nor does it involve the disorder associated to the $N^{th}$ site so that the interpolation employed in the cavity method does not affect this quantity.  First we have, by Lemmas \ref{L:Cav-Bound} and \ref{L:cavitation},
\begin{multline*}
\nu\left(\sqrt{N} \left(\varepsilon^1 \varepsilon^2- q_2\right) f\left(\EE_N^-\right)\right) =
\beta^2\left(1-q^2_2\right) \nu\left(\RR_{1,2}f\left(\EE_N^-\right)\right) \\-
\beta^2 2\left(q_2-q^2_2\right)\nu\left(\RR_{1,3}f\left(\EE_N^-\right)\right) \\
- \beta^2 2\left(q_2-q^2_2\right)\nu\left(\RR_{2,3}f\left(\EE_N^-\right)\right)\\+
\beta^2 3\left(q_4-q^2_2\right)\nu\left(\RR_{3,4}f\left(\EE_N^-\right)\right) +O\left(\frac{\|f\|_{\FF}}{\sqrt{N}}\right).
\\
=\beta^2\left(1-q_2\right)^2 \nu\left(\RR_{1,2}f\left(\EE_N^-\right)\right)
-\beta^2 b  \nu\left(\RR_{2,3}f\left(\EE_N^-\right)\right)
+O\left(\frac{\|f\|_{\FF}}{\sqrt{N}}\right).
\end{multline*}
In the last step we used replica symmetry along with the fact that $\EE_N^-$ only depends on the first replica to combine summands.  Since $\nu\left(\RR_{1,2}^2\right) \leq C$ and $\nu\left(\ell_N^2\right) \leq C$, we apply the Taylor expansion and Lemmas \ref{L:Cav-Bound} and \ref{L:cavitation} again to obtain
\begin{multline}
\label{Eq:ID-T-1}
\nu\left(\sqrt{N} \left(\varepsilon^1 \varepsilon^2- q_2\right) f\left(\EE_N^-\right)\right)= \\\beta^2\left(1-q_2\right)^2 \nu\left(\RR_{1,2}f\left(\EE_N\right)\right)
-\beta^2 b  \nu\left(\RR_{2,3}f\left(\EE_N\right)\right)
+O\left(\frac{\|f\|_{\FF}}{\sqrt{N}}\right).
\end{multline}

For the remaining term from \eqref{Eq:Start-1} we have (using Gaussian integration by parts)
\begin{multline*}
\nu\left(\left(\varepsilon^1 \varepsilon^2- q_2\right)\varepsilon^1 \ell^1_N f'\left(\EE_N^-\right)\right)=
\frac{\beta}{N}\sum_{j=2}^N \nu\left(\left(\varepsilon^1 \varepsilon^2- q_2\right)\varepsilon^1\sigma_j^1\left(\sum_{k=1}^2 \varepsilon^k\sigma_j^k \right) f'\left(\EE_N^-\right)\right)\\
- 2\frac{\beta}{N}\sum_{j=2}^N \nu\left(\left(\varepsilon^1 \varepsilon^2- q_2\right) \varepsilon^1\sigma_j^1\varepsilon^3\sigma_j^3 f'\left(\EE_N^-\right)\right).
\end{multline*}
We apply Lemma \ref{L:Cav-Bound}, the Taylor expansion, and symmetry to obtain
\begin{multline*}
\nu\left(\left(\varepsilon^1 \varepsilon^2- q_2\right)\varepsilon^1 \ell^1_N f'\left(\EE_N^-\right)\right)=
\beta \left(1-q_2\right)^2 \nu\left(R_{1,2} f'\left(\EE_N\right)\right) +\|f\|_\FF \frac{C}{\sqrt{N}}
\\
=\beta q_2 \left(1-q_2\right)^2 \nu\left(f'\left(\EE_N\right)\right)+\|f\|_\FF \frac{C}{\sqrt{N}}.
\end{multline*}
The last line here follows from the assumption that we are in the high temperature region.

Combining the various terms, we have shown that
\begin{equation*}
\nu\left(\left(1- \beta^2 a\right)\RR_{1,2}f (\EE_N)\right) =
-\beta^2 b\nu\left(\RR_{2,3}f (\EE_N)\right) + \beta q_2 a\nu\left(f' (\EE_N)\right)  + \|f\|_\FF \frac{C}{\sqrt{N}}.
\end{equation*}
\qed

\hspace{10pt}

A nearly identical argument allows us to treat the term involving $R_{2,3}$:
\begin{lemma}
\label{L:Point-Two}
We have the estimate
\begin{multline*}
\nu\left(\RR_{2,3}f (\EE_N)\right) = \beta^2 b \nu \left(\RR_{1,2}f\left(\EE_N\right)\right)
 +  \beta^2 c \nu \left( \RR_{2,3}f\left(\EE_N\right)\right) \\ +
 \beta q_2 b \nu\left(f'\left(\EE_N\right)\right) +\|f\|_\FF \frac{C}{\sqrt{N}}.
\end{multline*}
\end{lemma}
The above pair of lemmata give us a nontrivial system of equations allowing the solution of $\nu\left(\RR_{1,2}f (\EE_N)\right)$ in terms of $\nu\left(f' (\EE_N)\right) $:
\begin{corollary}
We have
\begin{multline*}
\left ((1- \beta^2a)(1-\beta^2c)+ \beta^4 b^2 \right) \nu\left(\RR_{1,2} f (\EE_N)\right)=\\
\beta q_2 \left(a - \beta^2 (b^2 + ac) \right) \nu\left(f' (\EE_N)\right)
+ \|f\|_\FF \frac{C}{\sqrt{N}}.
\end{multline*}
\end{corollary}
\textit{Proof.}
This is a straightforward manipulation using Lemmas \ref{L:Point-One} and \ref{L:Point-Two}.
\qed

\textit{Proof of Corollary \ref{C:WassAnn-1}.}
To translate Lemma \ref{L:Stein} to a Gaussian with variance $\sigma^2$, let $g(x)=\sigma g^*(\frac{x}{\sigma})$ and $f(x)= f^*(\sigma x)$, where $f^*$ solves the ordinary differential equation \eqref{eqOU} for $g^*$.  Then
\begin{equation}
\label{OUsigma}
\sigma^2{f}'(x) -x {f}'(x)= {g}(x) - \mathbb E\left[g(\sigma z)\right]
\end{equation}
where $z$ is a standard Gaussian variable.  Hence the bounds of Lemma \ref{L:Stein} become
\begin{equation}
\|f\|_{\infty} \leq L, \quad, \|{f}'\|_{\infty} \leq\sqrt{\frac {2}{ \pi}} \sigma L, \quad \|{f}''\|_{\infty} \leq2\sigma^2L
\end{equation}
where $L=Lip(g)= Lip(g^*)$.

Now let us fix $g: \mathbb R \rightarrow \mathbb R$ so that $Lip(g) \leq 1$ and consider
\begin{equation}
\label{A}
\left| \nu(g(\EE_N)) - \mathbb E\left[g(\sigma z)\right] \right|
\end{equation}
If $f$ solves \eqref{OUsigma}, then
\begin{equation}
\label{B}
\left| \nu(g(\EE_N)) - \mathbb E\left[g(\sigma z)\right] \right| = |\nu(\EE_N f(\EE_N)) - \sigma^2 \nu(f'(\EE_N))|.
\end{equation}
By Theorem \ref{T:annealed}, the right hand side is bounded by $C\|f\|_{\FF} N^{-\frac 12}$.  This estimate in turn is bounded by $CN^{-\frac 12}$ from the considerations of the previous paragraph.  Optimizing over $g$ gives the result. 
\qed

\section{The Quenched CLT for the Energy Density}
In this section we shall outline the more complex Theorem \ref{T:Main} and Corollary \ref{C:Quenched} which rely on elaborations of the ideas from the previous section.  We will postpone the proofs of some of the technical statements to Section \ref{S:Proofs}, in favor of giving a detailed sketch.  Most of the lemmata left unproven in this section are computationally intensive and obscure the main idea.

Our first goal is to derive approximate identities for the expressions on the left hand side of \eqref{eq:express}  with errors given in terms of the moments of $\|\XX^n\|_2^d$ under $\nu$ and $L^{\infty}$ norms of $F$ and its derivatives.
The first idea is to reduce their calculation to the calculation of the collection $\{\nu\left(\RR_{r, r'} G F \left(\XX^n\right)\right)\}_{r, r' \leq n+2}$.
Let $\vec{v}_i \in \mathbb R^{2n}$ be defined by
\begin{align*}
\vec{v}_1&:=\frac{1}{2} \left(1, q_2^2, \dotsc, q_2^2, 0, \dotsc, 0\right) \text{( $n$ zeros)}\\
\vec{v}_2&:=\beta q_2 \left( 0, \dotsc, 0, -q_1(1-q_2), q_1 + q_1q_2-2 q_3, \dotsc, q_1 + q_1q_2-2 q_3\right) \text{( $n$ zeros)}\\
\vec{v}_3&:=\frac{1}{2} \left(1, q_2^2, \dotsc, q_2^2, 0, \dotsc, 0\right) \text{( $n$ zeros)}.
\end{align*}
Next let $\pi$ denote the cyclic permutation $(1 \; 2 \dotsc n)$ and let $\pi \oplus \pi$ denote the product
of cycles $(1\;2 \dotsc n)(n + 1 \dotsc 2n)$. We let $\pi$ and $\pi \oplus \pi$ act naturally on the above
vectors by permuting coordinates, defining
\begin{align*}
\vec{e}^{\:i}&:= (\pi \oplus \pi)^{i-1} \cdot \vec{v}_1\\
\vec{m}^{i}&:= \left(\pi \oplus \pi \right)^{i-1} \left(\vec v_2 + \vec v_3\right)
\end{align*}

Further, let
$\vec{r^e}, \vec{r^m} \in \mathbb R^{{n+2 \choose 2}}$
\begin{equation*}
\vec{r}^e_{k, k'}:= \begin{cases}
1 \text{ if $k= 1, \; k' \leq  n$}\\
-n  \text{ if $k= 1, \; k' =  n+1$}\\
0 \text{ otherwise}
\end{cases}
\end{equation*}
and
\begin{equation*}
\vec{r}^m_{k, k'}:= \begin{cases}
q_1(1-q_2) \text{ if $k= 1, \; k' \leq  n$}\\
-nq_1(1-q_2)   \text{ if $k= 1, \; k' =  n+1$}\\
q_3- q_1 q_2 \text{ if $1< k \leq n, \: k' \leq n$}\\
-n(q_3- q_1 q_2) \text{ if $1< k \leq n, \: k' = n+1$}\\
{n+1 \choose 2} (q_3- q_1 q_2)  \text{ if $k = n+1, \: k' = n+2$}\\
0 \text{ otherwise.}
\end{cases}
\end{equation*}
Finally, we let $\vec \RR$ denote the ${n+2 \choose 2}$-tuple of lexicographically ordered overlaps
\begin{equation*}
\left(\RR_{1,2}, \cdots, \RR_{1,n+2}, \cdots,  \RR_{n+1, n+2}\right).
\end{equation*}

\begin{lemma}
\label{P:prep}
Suppose that the high temperature condition \eqref{Eq:High-Temp} holds.
Then we have the following identities:
\begin{equation}
\label{Eq:internal}
\nu(\EE^1 GF(\XX^n))= \nu\left(\vec {e^1} \cdot \nabla \left(GF\right)\right)+ \beta q_2 \nu\left(\vec {r^e} \cdot \vec \RR GF\right) + Er(F,G) \end{equation}
and
\begin{equation}
\label{Eq:magnet}
\nu\left(\MM^1 GF(\XX^n)\right)= \nu\left(\vec{m^1} \cdot \nabla \left(G F\right)\right)+ \beta^2 \nu\left(\vec{r^m} \cdot \vec \RR GF\right) + Er(F,G)
\end{equation}
\end{lemma}
The proof of Lemma \ref{P:prep} is left to Section \ref{S:Proofs}.

This lemma suggests that we should study the interaction between the overlaps and the function $GF(\XX^n)$.  This is accomplished in the next lemma, which can be seen as the main point of the paper.
Namely, the behavior of the fluctuations of macroscopic quantities in the SK model are controlled by their interactions with the overlaps.  Moreover, these interactions are computable through the derivation of an implicit set of equations.  This should not be too surprising, however we believe that this technique provides a useful general tool for proving CLTs for the high temperature phase of mean field disordered systems and thus deserves special attention.

For $1 \leq k < k' \leq n+2'$, let $A(k, k', r), \: B(k, k',r)$ be defined by
\begin{equation*}
A(k,k',r):=\nu_0\left(\left(\varepsilon^{k} \varepsilon^{k'}- q_2\right) \varepsilon^{r}\right)
\end{equation*}
and
\begin{equation*}
B(k,k',r):=
\begin{cases}
\beta q_2 b & \text{ if $k , k' \neq r$}\\
\beta a & \text{ if $k=r$ or  $k' = r$}\\
\end{cases}
\end{equation*}
otherwise.
For each pair $\{k, k'\}$, we denote by $\vec{w}_{k, k'}\in \mathbb R^{2n}$ the vector
\begin{equation*}
\vec{w}_{k, k'} :=\left(B(k,k', 1), \dotsc,  B(k,k', n),A(k, k',1), \dotsc, A(k, k',n),\right).
\end{equation*}
Let $\DD$ denote the \textit{vector} of first order differential operators $\DD:= (\vec{w}_{k,k'} \cdot \nabla)_{1 \leq k<k' \leq n+2}$.

Recall the definition of the matrix $\frak A$ from Section \ref{S:Results}.
\begin{lemma}
\label{L:overlap}
Suppose that the high temperature condition \eqref{Eq:High-Temp} holds.
Then entry-wise, we have the following identity:
\begin{equation}
\label{Eq:overlap-A}
\nu\left( GF(\XX^n) \left(Id - \beta^2 \frak A\right) \cdot \vec{\RR}\right) = \nu(\DD (GF)(\XX^n))+ Er(F,G).
\end{equation}
\end{lemma}
The proof of Lemma \ref{L:overlap} is left to Section \ref{S:Proofs}.

The question immediately arises as to when the matrix $Id-\beta^2 \frak A$ is invertible.  The condition of invertibility, combined with Lemmas \ref{P:prep} and \ref{L:overlap} allows us to obtain Theorem \ref{T:Main}:

\vspace{10pt}
\noindent\textit{Proof of Theorem \ref{T:Main}}
From Lemma \ref{L:overlap}, we have
\begin{equation*}
\nu(\vec{\RR} GF(\XX^n)) = \left(Id - \beta^2 \frak A \right)^{-1} \cdot \nu(\DD (GF)(\XX^n))+ Er(F,G).
\end{equation*}
Plugging this identity into the righthand side of equations \eqref{Eq:internal} and \eqref{Eq:magnet} we obtain the \text{existence} of a pair of vectors $\frak e, \frak m \in \mathbb R^{2n}$ so that
\begin{equation*}
\nu(\EE^1 GF(\XX^n))= \frak e \cdot \nu\left(\nabla \left(GF\right)\right)+ Er(F,G)
\end{equation*}
and
\begin{equation*}
\nu\left(\MM^1 GF(\XX^n)\right)= \frak m \cdot \nu\left(\nabla \left(G F\right)\right) + Er(F,G).
\end{equation*}

Now from the symmetry of replicas under $\nu$, we have
\begin{equation*}
\nu(\EE^i GF(\XX^n))= \frak e^i \cdot \nu\left(\nabla \left(GF\right)\right)+ Er(F,G)
\end{equation*}
and
\begin{equation*}
\nu\left(\MM^i GF(\XX^n)\right)= \frak m^i \cdot \nu\left(\nabla \left(G F\right)\right) + Er(F,G).
\end{equation*}
where
\begin{align*}
\frak e^i&:=
\pi^{i-1}\frak e\\
\frak m^i&:=
\pi^{i-1}\frak m.
\end{align*}
The covariance operator claimed in the theorem is then given by
\begin{equation*}
\frak C_{i, j}:=
\begin{cases}
\frak {e}^i_j \text{ if $i \leq n$}\\
 \frak {m}^{i-n}_j  \text{ if $i \geq n+1$}.
\end{cases}
\end{equation*}

By choosing $G$ to be one of the coordinate functions and $F$ to be the constant function $1$, the above equations allow us to identify $\frak C$ with the limiting covariance of $\XX^n_N$ under $\nu$.
\qed

\begin{proof}[Proof of Corollary \ref{C:Quenched}]
Let
\begin{equation*}
\sigma^2_N= \nu \left(\left(\HH_N-\langle\HH_N \rangle\right)^2\right)= \nu\left( \left(\HH^1_N-\HH^2_N\right)  \left(\HH^1_N-\HH^3_N \right) \right).
\end{equation*}
Now from Theorem \ref{T:Main}, $\sigma^2_N$ converges to a limit, which we denote by $\sigma_Q^2$, with a rate bounded by $CN^{-1/2}$.

We claim that
\begin{equation}
\label{eq:Ham}
\mathbb E \left[\left(
\left \langle \left(\HH_N-\langle\HH_N \rangle\right)^2 \right \rangle-\sigma_Q^2\right)^2 \right]  \leq \frac{C}{\sqrt{N}}.
\end{equation}
Indeed, using replicas
\begin{multline}
\label{eq:Ham-1}
\mathbb E \left[\left(
\left \langle \left(\HH_N-\langle\HH_N \rangle\right)^2 \right \rangle-\sigma_N^2\right)^2 \right]\\= \nu \left[\left( \left(\HH^1_N-\HH^2_N\right)\left(\HH^1_N-\HH^3_N\right)-\sigma_N^2\right)\left( \left(\HH^4_N-\HH^5_N\right)\left(\HH^4_N-\HH^6_N\right)-\sigma_N^2\right) \right].
\end{multline}
Now, under the quenched average measure $\nu$ the collection $\{\HH^\ell_N\}_{\ell=1} ^6$ are exchangeable and therefore Theorem \ref{T:Main} allows us to approximate the expression by the corresponding moment of a collection $\{h_\ell\}_{\ell=1}^6$ of exchangeable Gaussians, up to an error of order $\nu((\HH^1_N)^6)N^{-\frac{1}{2}}$ where  $\{h_{\ell}\}_{\ell=1}^6$ all have mean $0$, variance $\nu((\HH^1_N)^2)$ and the covariance between pairs is identically $\nu(\HH^1_N \HH^2_N)$.  
Also, a straightforward calculation with Gaussian integration by parts shows that under \eqref{Eq:High-Temp}, $\nu((\HH^1_N)^6) \leq C$ so that in fact this error is bounded by $CN^{-1/2}$.
Finally, note that \eqref{eq:Ham-1} vanishes if we insert $\{h_{\ell}\}_{\ell=1}^6$  in place of $\{\HH^\ell_N\}_{\ell=1} ^6$ because of the exchangeability, which proves the claim.
Since we observed in the previous paragraph that $|\sigma_N^2 - \sigma^2_Q| \leq \frac{C}{\sqrt{N}}$ the first assertion of the corollary follows.

We next show
\begin{equation*}
\mathbb E\left[\langle \HH_Nf\left(\HH_N\right)- \sigma_Q^2 f' \left(\HH_N\right) - \langle\HH_N\rangle f(\HH_N) \rangle^2 \right]=  \|f\|_\FF \; \nu\left((\HH_N)^4\right)^{1/2} O(N^{-1/2}).
\end{equation*}
Using replicas once again, we have
\begin{multline*}
\mathbb E\left[\langle \HH_Nf\left(\HH_N\right)- \sigma_Q^2 f' (\HH_N) - \langle\HH_N\rangle f(\HH_N) \rangle^2 \right] = \\\nu\left( \left\{(\HH^1_N-\HH_N^2)f(\HH^1_N)- \sigma_Q^2 f' (\HH^1_N)\right\}\left\{ (\HH^3_N- \HH^4_N)f(\HH^3_N)- \sigma_Q^2 f' (\HH^3_N) \right\}\right).
\end{multline*}
Via symmetry,
\begin{multline*}
\nu\left( \left\{(\HH^1_N-\HH_N^2)f\left(\HH^1_N\right)- \sigma_Q^2 f' (\HH^1_N)\right\}\left\{ (\HH^3_N- \HH^4_N)f(\HH^3_N)- \sigma_Q^2 f' (\HH^3_N) \right\}\right)=\\
\nu\left( (\HH^1_N-\HH_N^2)f(\HH^1_N)(\HH^3_N- \HH^4_N)f(\HH^3_N)\right)+ \sigma_Q^4 \nu\left( f' (\HH^1_N) f' (\HH^3_N)\right)\\
- 2 \sigma_Q^2 \nu\left( (\HH^1_N-\HH_N^2)f(\HH^1_N)f'(\HH^3_N)\right).
\end{multline*}
Let us apply Theorem \ref{T:Main} to the first term on the right hand side.  Letting $G(x)=x$ and $F(x, y)= f(x)f(y)$, we have $|\sigma_N^2- \sigma_Q^2|\leq CN^{-1/2}$ and
\begin{multline*}
\nu\left((\HH^1_N-\HH_N^2)f(\HH^1_N) (\HH^3_N- \HH^4_N) f(\HH^3_N)\right)=\\
\sigma_Q^2 \nu\left((\HH^1_N-\HH_N^2)f(\HH^1_N)f'(\HH^3_N)\right) + Er(x,  F).
\end{multline*}

This leaves us with
\begin{multline*}
\mathbb E \left[\langle \HH_Nf(\HH_N)- \sigma_Q^2 f' (\HH_N) - \langle\HH_N\rangle f(\HH_N) \rangle^2 \right] = \\
\sigma_Q^2 \left\{\sigma_Q^2 \nu\left( f' (\HH^1_N) f' (\HH^3_N)\right)- \nu\left((\HH^1_N-\HH_N^2)f(\HH^1_N)f'(\HH^3_N)\right) \right\}\\ + Er(x,  F)
\end{multline*}
Iterating the previous argument yields
\begin{equation*}
\mathbb E \left[\langle \HH_Nf\left(\HH_N\right)- \sigma_Q^2 f' \left(\HH_N\right) - \langle\HH_N\rangle f(\HH_N) \rangle^2 \right] = \\
Er(x,  F) + Er(1, \nabla F).
\end{equation*}
\end{proof}


\section{Proof of Theorem \ref{T:Sum}}

\textit{Proof of Theorem \ref{T:Sum} \eqref{I:1}.}
We apply Theorem \ref{T:Main} in the case $n=1$ and for functions
\begin{equation*}
F(\EE_N, \MM_N):= f(\EE_N + \MM_N).
\end{equation*}
With $w = (1, 1)$  this gives
\begin{equation*}
\nu\left(\HH_N f(\HH_N)\right) = (w, \frak C w) \nu\left(f' (\HH_N)\right) + C\|f\|_\FF N^{-1/2}
\end{equation*}
where $\frak C$ is the limiting quenched average covariance matrix for $\MM_N, \EE_N$.
We are thus in a position to directly apply Stein's Method.
\qed
\vspace{10pt}

\textit{Proof of Theorem \ref{T:Sum} \eqref{I:2}.}
This statement is a consequence of a more general principle, which may be of independent interest.  

Let us consider the following abstract setup: we are given a sequence of random variables $(X_N, \mu_N, \sigma_N)_{N\ge 1}$ on a probability space $(P, \Omega, \FF)$ and let $(\mf_N)_{N\ge 1}$ be a sequence of sub-sigma fields of $\FF$ such that for each $N$, $\mu_N$ and $\sigma_N$ are measurable with respect to $\mf_N$ (in particular, we do not assume $\mu_N$, $\sigma_N$ are deterministic. Let $\ee^N$ denote the conditional expectation given $\mf^N$.

For any $f\in C^k(\rr)$, let
\[
|f|_k := \|f\|_\infty + \|f^{(1)}\|_\infty + \cdots + \|f^{(k)}\|_\infty,
\]
where $f^{(m)}$ denotes the $m$th derivative of $f$.
\begin{lmm}
\label{L:Sour}
Suppose that for some fixed integer $k$ and positive real numbers $\alpha$ and $L$, we have that for each $N$ and each $f\in C^k(\rr)$,
\[
\ee\bigl(\ee^N(\sigma_N^2f'(X_N) - (X_N -\mu_N)f(X_N))\bigr)^2 \le L|f|_k N^{-\alpha}.
\]
Suppose that for each $\theta\ge 0$, $\ee(e^{\theta \sigma_N^2})$ is uniformly bounded over $N$. Let $\nu_N$ be the conditional law of $X_N$ given $\mf_N$, and $\gamma_N$ be the Gaussian measure with mean $\mu_N$ and variance $\sigma_N^2$. Then the random  signed measure $\nu_N - \gamma_N$ converges to the zero measure in probability as a random sequence on  the space of finite signed measures with the metric of weak* convergence.
\end{lmm}

Nothing like the full power of Lemma \ref{L:Sour} is required here.  In fact, we know by \eqref{eq:Ham} that we may take
\begin{equation*}
\sigma^2_N \equiv \sigma^2_Q := \lim_{K \rightarrow \infty} \mathbb E \left[ \langle \left(\HH_K-\langle\HH_K \rangle\right)^2 \rangle \right].
\end{equation*}
Thus the second statement of Theorem \ref{T:Sum} follows immediately Lemma \ref{L:Sour} combined with Corollary \ref{C:Quenched} and the fact that the weak* topology, when restricted to probability measures, is metrizable by $\rho$.
\vspace{10pt}

\textit{Proof.}
Taking $f(x) = e^{itx} = \cos tx + i \sin tx$, where $i = \sqrt{-1}$, we can apply the above inequality separately to the real and imaginary parts to get
\[
\ee\bigl(\ee^N(\sigma_N^2it e^{itX_N} - (X_N -\mu_N)e^{itX_N})\bigr)^2 \le 2(k+1)L\max\{1, |t|^k\} N^{-\alpha}.
\]
Define the random function
\[
\phi_N(t) := \ee^N(e^{it(X_N - \mu_N)}).
\]
Since $\mu_N$ and $\sigma_N^2$ are measurable with respect to $\mf_N$ and $|ie^{it\mu_N}|=1$, we have
\begin{align*}
&\bigl|\ee^N(\sigma_N^2it e^{itX_N} - (X_N -\mu_N)e^{itX_N})\bigr|\\
&= \bigl|\ee^N(\sigma_N^2t e^{it(X_N-\mu_N)} + i(X_N -\mu_N)e^{it(X_N-\mu_N)})\bigr|\\
&= \bigl|\sigma_N^2 t\phi_N(t) + \phi_N'(t)\bigr|.
\end{align*}
Let $\psi_N(t) := e^{-\sigma_N^2t^2/2}$. Then $\phi_N(0)=\psi_N(0)=1$, and $\psi_N'(t) = - \sigma_N^2 t \psi_N(t)$. Thus, for all $t\ge 0$, we have
\begin{align*}
|\phi_N(t)-\psi_N(t)| &\le \int_0^t|\phi_N'(s)-\psi_N'(s)| ds\\
&\le \int_0^t \sigma_N^2 s |\phi_N(s) -\psi_N(s)| ds + \int_0^t |\sigma_N^2 s \phi_N(s) + \phi_N'(s)| ds.
\end{align*}
Now fix $t\ge 0$, and let
\[
A := \sigma_N^2 t, \ \ B := \int_0^t |\sigma_N^2 s \phi_N(s) + \phi_N'(s)| ds.
\]
Also, let $v(s) := |\phi_N(s) - \psi_N(s)|$. Then we see from the last inequality that for all $s\in [0, t]$,
\[
v(s) \le A\int_0^s v(u) du + B.
\]
By the standard method of using the bound recursively, we get that for all $s\in[0,t]$,
\[
v(s) \le Be^{As}.
\]
Combining the steps and using the Cauchy-Schwarz inequality, we get
\begin{align*}
\ee|\phi_N(t) - \psi_N(t)|&\le \ee\biggl(e^{\sigma_N^2 t^2} \int_0^t |\sigma_N^2 s \phi_N(s) + \phi_N'(s)| ds\biggr)\\
&\le \int_0^t \bigl(\ee(e^{2\sigma_N^2t^2})\bigr)^{1/2} \bigl(\ee(\sigma_N^2 s \phi_N(s) + \phi_N'(s))^2\bigr)^{1/2} ds\\
&\le C(t) N^{-\alpha/2},
\end{align*}
where $C(t)$ is a constant depending only on $t$. This shows that $\ee|\phi_N(t) - \psi_N(t)|\ra 0$ as  $N\ra \infty$ for every $t\ge 0$. The same result holds for $t\le 0$ as well. Since $\phi_N(t)-\psi_N(t)$ is the characteristic function of the signed measure $\nu_N - \gamma_N$, the proof can now be completed using standard `subsequence-of-subsequence' arguments.
\qed
\vspace{10pt}

\textit{Proof of Theorem \ref{T:Sum} \eqref{I:3}.}
To prove this statement, we compute $\phi_N(t):=\mathbb E \left[e^{it \langle \HH_N \rangle} \right]$.
By Theorem \ref{T:Sum} \eqref{I:1} and standard results on characteristic functions
\begin{equation*}
\lim_{N \rightarrow \infty} \nu\left(e^{it \HH_N} \right) = e^{-\sigma_A^2 t^2/2}
\end{equation*}
On the other hand, the proof of Lemma \ref{L:Sour} gives
\begin{equation*}
\lim_{N \rightarrow \infty} \mathbb E \left[ \left| \langle e^{it \left( \HH_N - \langle \HH_N \rangle \right)} - e^{-\sigma_Q^2 t^2/2} \rangle \right | \right] = 0
\end{equation*}
Since
\begin{equation*}
 \nu\left(e^{it \HH_N} \right) =  \mathbb E \left[ \langle e^{it (\HH_N- \langle \HH_N \rangle)}\rangle e^{it \langle \HH_N \rangle} \right]
\end{equation*}
we obtain
\begin{equation*}
\lim_{N \rightarrow \infty} \phi_N(t) = e^{-(\sigma_A^2- \sigma_Q^2) t^2/2}.
\end{equation*}
\qed


\section{Proofs of Lemmata}
\label{S:Proofs}
\textit{Proof of Lemma \ref{L:cavitation}}
Note that under the hypothesis of the lemma,
\begin{equation*}
\nu_0\left(\left(\prod_{l=1}^p \varepsilon^{\eta\left(l\right)}- q_p\right)f^{-}\right)=0.
\end{equation*}
Using \eqref{Cav-2},
substituting from \eqref{Cav-Method} and using the decoupling property of $\nu_0$,
{\small \begin{multline*}
\nu\left(\left(\prod_{r=1}^p \varepsilon^{\eta\left(r\right)}- q_p\right)f^{-}\right)=\beta^2 \left(q_{p-2}- q_pq_2\right)\sum_{\{r, r'\} \subseteq [p]} \nu_0\left(f^-\left(R^-_{\eta(r), \eta(r')}-q_2\right)\right)\\
+\beta^2 \left(q_p- q_p q_2\right)\sum_{\underset{1 \leq r \leq p}{ r' \in [n] \backslash \eta\left([p]\right)}} \nu_0\left(f^-\left(R^-_{\eta(r), r'}-q_2\right)\right)\\
+ \beta^2  \left(q_{p+2}-q_pq_2\right)\sum_{\{r, r'\} \subset [n] \backslash \eta([p]) = \varnothing} \nu_0\left(f^-\left(R^-_{\eta(r), n+1}-q_2\right)\right)\\
- \beta^2 n \left(q_p-q_pq_2\right)\sum_{1 \leq r \leq p} \nu_0\left(f^-\left(R^-_{\eta(r), n+1}-q_2\right)\right)\\
- \beta^2 n \left(q_{p+2}-q_pq_2\right)\sum_{ r \in [n] \backslash \eta([p])} \nu_0\left(f^-\left(R^-_{r, n+1}-q_2\right)\right)\\
+  \beta^2 \frac{n\left(n+1\right)}{2} \left(q_{p+2}-q_p q_2\right)\nu_0\left(f^-\left(R^-_{n+1, n+2}-q_2\right)\right) + \nu_0\left((f^-)^2\right)^{\frac{1}{2}}\frac{C}{N}.
\end{multline*}
}
Recalling that $R_{k, k'}^-= R_{k, k'}- \frac{\varepsilon^{k} \varepsilon^{k'}}{N}$, we may replace each instance of $R_{k, k'}^-$ by $R_{k, k'}$ at the cost of an error of the form $ \nu_0\left((f^-)^2\right)^{\frac{1}{2}}CN^{-1}$.   Applying \eqref{Cav-1}, using the Cauchy-Schwarz inequality and the high temperature condition \eqref{Eq:High-Temp} to bound the error gives the result.
\qed

\begin{proof}[Proof of Corollary \ref{C:overlap}]
We shall present the calculation only in the case of $\nu\left(\RR_{n+1,n+2} F\right)$.
The identities for the remaining coordinates are straight forward adaptations.
By definition
\begin{equation*}
\nu\left(\RR_{n+1,n+2} F\right)=\frac{1}{\sqrt{N}} \sum_{i=1}^N \nu\left(\left(\sigma_i^{n+1} \sigma_i^{n+2}- q_2 \right) F(\VV^n)\right).
\end{equation*}

Let us restrict attention to a single summand as each of the summands may be treated similarly.
Using the mean value theorem to expand $F$ around $\VV^{n}_i$ up to its second order derivatives, we have
\begin{multline*}
\nu\left(\left(\sigma_i^{n+1} \sigma_i^{n+2}- q_2 \right) F(\VV^n)\right)= \nu\left(\left(\sigma_i^{n+1} \sigma_i^{n+2}- q_2 \right) F(\VV^{n}_i)\right)\\
+\nu\left( \left(\sigma_i^{n+1} \sigma_i^{n+2}- q_2 \right) \nabla F(\VV^n_i)\cdot(\VV^n-\VV_i^n)\right)\\+
\frac{1}{2} \nu\left( \left(\sigma_i^{n+1} \sigma_i^{n+2}- q_2 \right) (\VV^n-\VV_i^n) \cdot \nabla^2 F(\lambda_0\VV^n + (1- \lambda_0)\VV^n_i)(\VV^n-\VV_i^n)\right)
\end{multline*}
where
the constant $\lambda_0$ may be taken to depend measurably on $\VV^n, \VV^n_i$.

Obviously, the interpolation scheme applies equally to any vertex $i$, so Lemma \ref{L:cavitation} implies that
\begin{equation*}
\nu\left(\left(\sigma_i^{n+1} \sigma_i^{n+2}- q_2 \right) F(\VV^{n}_i)\right)= \frac{\beta^2}{\sqrt N} \sum_{1 \leq r< r' \leq n+4}  \tilde{C}_{r,r'} \nu\left( \RR_{r, r'} F (\VV^{n}_i)\right) + \nu\left(F^2 (\VV^{n}_i)\right)^{\frac{1}{2}} O(N^{-1}).
\end{equation*}
where
\begin{equation*}
 \tilde{C}_{r,r'} =
 \begin{cases}
 q_4- q_2^2 & \text{ if $r, r' \leq n$}\\
q_2- q_2^2 & \text{ if $r\leq n, r' \in \{n+1, n+2\}$}\\
 1- q_2^2 & \text{ if $r=n+1, r' = n+2$}\\
 -n(q_4- q_2^2)  & \text{ if $r\leq n, r' = n+3$}\\
 -n (q_2- q_2^2) & \text{ if $r \in \{n+1, n+2\}, r' = n+3$} \\
  \frac{(n+2)(n+3)}{2} (q_4-q_2^2) & \text{ if $r=n+3, r' = n+4$}.
 \end{cases}
\end{equation*}

Because $F$ depends only on the first $n$ replicas, we may use the exchangeability of replicas to rewrite this equation in terms of the overlaps from the first $n+2$ replicas.  We have (with the notation \eqref{Eq:Replica-Matrix})
\begin{multline*}
\nu\left(\left(\sigma_i^{n+1} \sigma_i^{n+2}- q_2 \right) F(\VV^{n}_i)\right)= \\
\frac{\beta^2}{\sqrt N} \sum_{1 \leq r< r' \leq n+2}  \frak{A}^{n+1, n+2}_{r,r'} \nu\left( \RR_{r, r'} F (\VV^{n}_i)\right) + \nu\left(F^2 (\VV^{n}_i)\right)^{\frac{1}{2}} O(N^{-1}).
\end{multline*}
Expressing $\nu\left( \RR_{r, r'} F (\VV^{n}_i)\right) = \nu\left( \RR_{r, r'} F (\VV^{n})\right) + \nu\left( \RR_{r, r'} F (\VV^{n}_i)\right) - \nu\left( \RR_{r, r'} F (\VV^{n})\right) $ and collecting terms finishes the proof.
\end{proof}

By Corollary \ref{C:overlap}, to  prove Lemma \ref{L:overlap} we have two tasks: compute derivatives and bound errors. We begin by isolating a series of preparatory calculations.
Let us denote the local field for the $N$'th site and $r$'th replica by $\ell_{N}^r$ and let
\begin{equation*}
\vec{\psi}=\vec \psi_N^n:= (\varepsilon^1\ell_N^1, \varepsilon^1,\dotsc,  \varepsilon^n\ell_N^n, \varepsilon^n)
\end{equation*}
We first note that $\|\vec{\psi}\|_2^2$ has finite moments of all orders.
\begin{lemma}
Let $k \in \mathbb N$ be fixed.  Then there exists a $C(k)$ depending only on $k$ so that
\begin{equation*}
\nu\left((\ell_N^1)^{2k}\right) \leq C(k)
\end{equation*}
\end{lemma}
\begin{proof}
The proof is by Gaussian integration by parts and induction on $k$.  The reader may consult \cite{Chatterjee} for more details.
\end{proof}

Let us introduce
\begin{equation*}
\XX^{n, -} := \XX^n- \frac{1}{\sqrt{N}}\vec{\psi}.
\end{equation*}
\begin{lemma}
\label{L:Rep}
Suppose that the high temperature condition \eqref{Eq:High-Temp} holds.  Let $l \in [2n]$ be fixed.  Then we have
\begin{equation}
\label{Eq:overlap-A-2}
\nu\left((\varepsilon^{k} \varepsilon^{k'}- q_2) \varepsilon^{r} \partial_{x_{l}}(G F)(\XX^{n,-})\right)= A(k,k', r) \nu(\partial_{x_{j}}(G F)(\XX^n))+ Er(F,G).
\end{equation}
\end{lemma}
\begin{proof}
This follows from a straightforward sequence of applications of \eqref{Cav-1} and Taylor's Theorem.
\end{proof}

\begin{lemma}
\label{L:Field-Loc}
Suppose that the high temperature condition \eqref{Eq:High-Temp} holds.  Let $l \in [2n]$ be fixed and suppose that $r \leq n$.
Then we have
\begin{equation}
\label{Eq:overlap-b}
\nu\left((\varepsilon^{k}\varepsilon^{ k'}- q_2) \varepsilon^{r}\ell_N^r \partial_{x_{l}} (G F)(\XX^{n,-}))= B(k,k', r) \nu(\partial_{x_j}(GF)(\XX^n)\right)+ Er(F,G).
\end{equation}
\end{lemma}

\begin{proof}
We consider only the generic case $1 \leq k< k' \leq n, \; k,k' \neq r$.  The remaining cases are easy adaptations of the main argument.
The first step is to apply Gaussian integration by parts.  Noting that $\XX^{n,-}$ does not involve the disorder appearing in the local field under consideration,
\begin{multline*}
\nu\left((\varepsilon^k \varepsilon^{k'}- q_2) \varepsilon^r \ell_N^r \partial_{x_{l}}  \left(GF\right) (\XX^{n, -})\right) = \\
\sum_{1 \leq r' \leq n} \sum_{j=1}^{N-1} \frac{\beta}{N} \nu\left((\varepsilon^k\varepsilon^{k'}- q_2) \varepsilon^r\sigma_j^r \varepsilon^{r'} \sigma_j^{r'} \partial_{x_{l}} (GF) (\XX^{n,-})\right)\\
- n\sum_{j=1}^{N-1} \frac{\beta}{N} \nu\left((\varepsilon^k\varepsilon^{k'}- q_2) \varepsilon^r\sigma_j^{r} \varepsilon^{n+1} \sigma_j^{n+1} \partial_{x_{l}} (GF) (\XX^{n,-})\right).
\end{multline*}
Rewriting in the language of overlaps we have
\begin{multline*}
\nu\left((\varepsilon^k \varepsilon^{k'}- q_2) \varepsilon^r \ell_N^r \partial_{x_{l}}  \left(GF\right) (\XX^{n, -})\right) = \\
\sum_{1 \leq r' \leq n} \beta \nu\left((\varepsilon^k \varepsilon^{k'}- q_2) \varepsilon^r \varepsilon^{r'} R_{r,r'}^- \partial_{x_{l}} (GF) (\XX^{n,-})\right)\\
- n\beta \nu\left((\varepsilon^k \varepsilon^{k'}- q_2) \varepsilon^r \varepsilon^{n+1} R_{r,n+1}^- \partial_{x_{l}} (GF) (\XX^{n,-})\right).
\end{multline*}

Successive applications of \eqref{Cav-1} and Taylor's theorem (cf. Section 3) give
\begin{multline*}
\nu\left((\varepsilon^k \varepsilon^{k'}- q_2) \varepsilon^r \ell_N^r \partial_{x_{l}}  \left(GF\right) (\XX^{n, -})\right) =  \beta q_2 [2q_2+q^2_2 - 3 q_4]\nu\left(\partial_{x_{l}} (GF) (\XX^{n})\right) \\+ Er(F,G).
\end{multline*}
The extra factor of $q_2$ in front comes from the replacement of the variables $R_{k, k'}$ using assumption \eqref{Eq:High-Temp}
\end{proof}

\textit{Proof of Lemma \ref{L:overlap}.}
Applying Corollary \ref{C:overlap} to the function $K(\VV)$ defined implicitly by
\begin{equation*}
K(\VV):= GF(\XX^n),
\end{equation*}
we must compute an expression for the first term on the righthand side.
Note that
\begin{equation*}
\frac{\partial}{\partial \sigma_i^r}K = \frac{1}{\sqrt N} \left(\ell^r_i \frac{\partial}{\partial x_r}GF(\XX^n_i)+ \frac{\partial}{\partial x_{r+n}}GF(\XX^n_i)\right).
\end{equation*}
Thus, we may compute
$\nu\left((\sigma^r_i\sigma^{r'}_i-q_2)\nabla K(\VV^n_i)\cdot (\VV^n-\VV^n_i)\right)$ by directly applying Lemmas \ref{L:Rep} and \ref{L:Field-Loc}.  

To bound the error given by Corollary \ref{C:overlap}, we use Taylor's Theorem and H\"{o}lder's inequality along with \eqref{Eq:High-Temp} to obtain
\begin{equation*}
|Rem(K)| \leq Er(F, G).
\end{equation*}
\qed

\vspace{10pt}

\textit{Proof of Lemma \ref{P:prep}.}
The first equation, \eqref{Eq:internal}, follows from a calculation using integration by parts.  Indeed, we have
\begin{multline*}
\nu\left(\sum_{i < j} \frac{\beta}{N} g_{i,j}\sigma^i \sigma^j G(\XX^n)F(\XX^n)\right) = \frac{1}{N^2}\sum_{i<j} \sum_{1 \leq r \leq n}\nu\left(\sigma_i^1 \sigma_j^1\sigma_i^r \sigma_j^r \partial_{x_{r}} \left(GF\right)\right) \\+ \frac{\beta}{N^{3/2}}\sum_{i<j} \sum_{1 \leq r \leq n}\nu\left(\sigma_i^1 \sigma_j^1\sigma_i^r \sigma_j^rGF\right) \\
- n \frac{\beta}{N^{3/2}}\sum_{i<j} \nu\left(\sigma_i^1 \sigma_j^1\sigma_i^{n+1} \sigma_j^{n+1} GF\right).
\end{multline*}
Using the language of overlaps, this simplifies to
\begin{multline*}
\nu\left(\sum_{i < j} \frac{\beta}{N} g_{i,j}\sigma^i \sigma^j G(\XX^n)F(\XX^n)\right) = \frac{1}{2}\sum_{1 \leq r \leq n}\nu\left(R^2_{1,r} \partial_{x_{2r-1}} \left(GF\right)\right) \\+ \frac{\beta \sqrt{N}}{2}\sum_{1 \leq r \leq n}\nu\left(R_{1,r}^2GF\right) - n\frac{\beta \sqrt N }{2} \nu\left(R_{1, n+1}^2 GF\right) + Er(F, G)
\end{multline*}
where $R_{1,1}=1$.
Let us expand each of the overlaps (excluding $R_{1,1}$) using
\begin{equation*}
R_{1,r}= R_{1,r}-q_2+q_2.
\end{equation*}
By \eqref{Eq:High-Temp} and the Cauchy-Schwarz inequality we have
\begin{multline*}
\nu\left(\sum_{i < j} \frac{\beta}{N} g_{i,j}\sigma^i \sigma^j G(\XX^n)F(\XX^n)\right) =\\
 \frac{1}{2} \nu\left(\partial_{x_{1}} \left(GF\right)\right)+ \frac{\beta \sqrt{N}(1-q_2^2)}{2} \nu\left(GF\right)  + \frac{q_2^2}{2}\sum_{2 \leq r \leq n}\nu\left(\partial_{x_{r}} \left(GF\right)\right) \\
  +  \beta q_2 \sum_{2 \leq r \leq n}\nu\left(\RR_{1,r}GF\right) - n \beta q_2 \nu\left(\RR_{1, n+1} GF\right)+Er(F,G),
\end{multline*}
which proves the first statement.

Let us consider next the term involving the scaled magnetization.
By symmetry,
\begin{equation*}
\nu\left(\MM^1 G(\XX^{n})F (\XX^{n})\right)= \nu\left(\sqrt{N}(\varepsilon^1- q_1) G(\XX^{n})F (\XX^{n})\right).
\end{equation*}
We have
\begin{multline*}
\nu\left(\sqrt{N}\right(\varepsilon^{1} - q_1\left) G(\XX^n) F (\XX^n)\right)= \nu\left(\sqrt{N} \left(\varepsilon^1 -q_1\right) (G F) (\XX^{n,-})\right)\\
+\sum_{1 \leq r \leq n} \nu\left( \left(\varepsilon^1 -q_1\right) \varepsilon^r \ell_N^r \partial_{x_{r}} (GF) (\XX^{n,-})\right)\\+ \sum_{1 \leq r \leq n} \nu\left( \left(\varepsilon^1 -q_1\right)  \varepsilon^r\partial_{x_{n+r}} (GF) (\XX^{n,-})\right)
+ Er(F,G).
\end{multline*}

Applying Lemma \ref{L:Cav-Bound} and noting that the zeroth order term for the decoupled measure vanishes,
\begin{multline*}
\nu\left(\sqrt{N}(\varepsilon^1- q_1) G(\XX^{n, -})F (\XX^{n, -})\right)=\\
\nu_{0}'(\sqrt{N} (\varepsilon^1- q_1) G(\XX^{n, -})F(\XX^{n, -})) +  Er(F,G).
\end{multline*}
Using symmetry, Lemma \ref{L:cavitation} gives
\begin{multline*}
\nu_{0}'\left( \varepsilon^1- q_1) G(\XX^{n, -}) F (\XX^{n, -})\right)= \\
\beta^2q_1(1-q_2)\sum_{2 \leq r \leq n} \nu\left(G(\XX^{n,-}) F (\XX^{n,-})\RR_{1,r}\right) \\+\beta^2(q_3-q_1q_2) \sum_{\{r, r'\} \subseteq [n]\backslash \{1\}} \nu(G(\XX^{n,-}) F (\XX^{n,-})\RR_{r,r'}) \\
-n \beta^2q_1(1-q_2)\nu(G(\XX^{n,-}) F (\XX^{n,-})\RR_{1,n+1})\\
 - n\beta^2(q_3-q_1q_2)\sum_{2 \leq r \leq n} \nu(G(\XX^{n,-}) F (\XX^{n,-})\RR_{r,n+1}) \\+\frac{n(n+1)}{2} \beta^2(q_3-q_1q_2)\nu(G(\XX^{n,-}) F (\XX^{n,-})\RR_{n+1,n+2}) + Er(F, G).
\end{multline*}
An application of Taylor's theorem and the Cauchy-Schwarz inequality implies
\begin{multline*}
\nu_{0}'( \varepsilon^1- q_1) G(\XX^{n,-} )F (\XX^{n,-}))= \beta^2q_1(1-q_2)\sum_{2 \leq r \leq n} \nu(G(\XX^{n}) F (\XX^{n})\RR_{1,r}) \\+\beta^2(q_3-q_1q_2) \sum_{\{r, r'\} \subseteq [n]\backslash \{1\}} \nu(G(\XX^{n}) F (\XX^{n})\RR_{r,r'}) \\
-n \beta^2q_1(1-q_2)\nu(G(\XX^{n}) F (\XX^{n})\RR_{1,n+1})\\
 - n\beta^2(q_3-q_1q_2)\sum_{2 \leq r \leq n} \nu(G(\XX^{n}) F (\XX^{n})\RR_{r,n+1}) \\+\frac{n(n+1)}{2} \beta^2(q_3-q_1q_2)\nu(G(\XX^{n}) F (\XX^{n})\RR_{n+1,n+2})\\
 + Er(F,G).
\end{multline*}

Next let us consider the terms given by the derivatives of $GF$.
Arguments analogous to Lemmas \ref{L:Rep} and \ref{L:Field-Loc} show that we have (using Lemma \ref{L:Cav-Bound} and Taylor's Theorem)
\begin{equation*}
\nu\left((\varepsilon^1-  q_1) \varepsilon^1 \partial_{x_{n+1}}(GF) (\XX^{n,-})\right)=(1 - q_1^2) \nu\left(\partial_{x_{n+1}} (GF) (\XX^{n})\right) +  Er(F,G)
\end{equation*}
and
\begin{equation*}
\nu\left((\varepsilon^1-  q_1) \varepsilon^r \partial_{x_{n+r}} (GF) (\XX^{n,-})\right)= (q_2 - q_1^2) \nu\left(\partial_{x_{n+r}}(G F) (\XX^{n})\right)+  Er(F,G)
\end{equation*}
for $r \geq 2$.

The terms involving the local fields $\ell^1_N, \dotsc,  \ell^n_N$ may be computed using integration by parts.  For each $k \in \{1, \dotsc, n\}$, we have (noting that $\XX^{n,-}$ does not involve the disorder appearing in the local fields under consideration)
\begin{multline*}
\nu\left((\varepsilon^1- q_1) \varepsilon^k \ell_N^k \partial_{x_{k}}  \left(GF\right) (\XX^{n, -})\right) = \\
\sum_{1 \leq r \leq n} \sum_{j=1}^{N-1} \frac{\beta}{N} \nu\left((\varepsilon^1- q_1) \varepsilon^k\sigma_j^k \varepsilon^r \sigma_j^r \partial_{x_{k}} (GF) (\XX^{n,-})\right)\\
- n\sum_{j=1}^{N-1} \frac{\beta}{N} \nu\left((\varepsilon^1- q_1) \varepsilon^k\sigma_j^k \varepsilon^{n+1} \sigma_j^{n+1} \partial_{x_{k}} (GF) (\XX^{n,-})\right).
\end{multline*}
Rewriting in the language of overlaps we have
\begin{multline*}
\nu\left((\varepsilon^1- q_1) \varepsilon^k \ell_N^k \partial_{x_{k}}  \left(GF\right) (\XX^{n, -})\right) = \\
\sum_{1 \leq r \leq n} \beta \nu\left((\varepsilon^1- q_1) \varepsilon^k \varepsilon^r R_{k,r}^- \partial_{x_{k}} (GF) (\XX^{n,-})\right)\\
- n\beta \nu\left((\varepsilon^1- q_1) \varepsilon^k \varepsilon^{n+1} R_{k,n+1}^- \partial_{x_{k}} (GF) (\XX^{n,-})\right).
\end{multline*}

The usual applications of Lemma \ref{L:Cav-Bound} and Taylor's theorem give
\begin{multline*}
\nu\left((\varepsilon^1- q_1) \varepsilon^1 \ell_1^1 \partial_{x_{1}}  \left(GF\right) (\XX^{n, -})\right) =\sum_{\substack{2 \leq  r \leq n}} \beta q_1(1-q_2) \nu\left(R_{1,r} \partial_{x_{1}} (GF) (\XX^{n})\right)\\
- n\beta q_1(1-q_2) \nu\left(R_{1,n+1} \partial_{x_{1}} (GF) (\XX^{n})\right) + Er(F,G)\\
= - \beta q_1q_2 (1-q_2) \nu\left( \partial_{x_{1}} (GF) (\XX^{n})\right) + Er(F,G).
\end{multline*}
and
\begin{multline*}
\nu\left((\varepsilon^1- q_1) \varepsilon^k \ell_1^k \partial_{x_{k}}  \left(GF\right) (\XX^{n, -})\right) =\beta q_1(1-q_2) \nu\left(R_{1,k} \partial_{x_{k}} (GF) (\XX^{n})\right) +\\
\sum_{\substack{2 \leq  r \leq n \\ r \neq k}} \beta(q_3-q_2 q_1) \nu\left(R_{k,r} \partial_{x_{k}} (GF) (\XX^{n})\right)\\
- n\beta(q_3-q_2q_1) \nu\left(R_{k,n+1} \partial_{x_{k}} (GF) (\XX^{n})\right) + Er(F,G)\\
= \beta q_1q_2(1-q_2) \nu\left(\partial_{x_{k}} (GF) (\XX^{n})\right) \\
- 2 \beta q_2(q_3-q_2q_1) \nu\left(\partial_{x_{k}} (GF) (\XX^{n})\right) + Er(F,G)
\end{multline*}
for each $2 \leq k \leq n$.
The second equality in each expression follows by replacing each overlap by $q_2$ at the cost of a term of type $Er(F,G)$ via \eqref{Eq:High-Temp}.
Collecting terms gives the final expression.
\qed

\section{Acknowledgments}
The authors would like to thank M.~Talagrand for suggesting the study of fluctuations for $H_N$ in non-zero external field.  N.~Crawford would like to further thank P.L.~Contucci and C.~Giardina for an invitation to the March 2007 Young European Probabilists workshop where this work was first presented.  In addition, he would like to acknowledge interesting discussions with S.~Starr on this and related topics during this meeting.

\end{document}